\documentclass[reqno]{amsart}

\usepackage{amsfonts}
\usepackage{amssymb}
\usepackage{amsmath}
\usepackage{mathrsfs}
\usepackage{amsthm}
\usepackage[colorlinks,linkcolor=blue,citecolor=blue]{hyperref}
\usepackage{enumitem}

\usepackage{mathtools}
\usepackage{soul}

\allowdisplaybreaks[4] % 允许公式环境分页

\newtheorem{theorem}{Theorem}[section]

\newtheorem{lemma}[theorem]{Lemma}
\newtheorem{proposition}[theorem]{Proposition}

\theoremstyle{remark}
\newtheorem{remark}[theorem]{Remark}
\renewcommand{\H}{{\mathbb{H}}}

\newcommand{\pd}[2]{\frac{\partial {#1}}{\partial {#2}}}

\theoremstyle{definition}
\newtheorem{definition}[theorem]{Definition}

\numberwithin{equation}{section}

\begin{document}

\title[Inverse curvature flow]
{Inverse Hessian Curvature Flow in Minkowski Space II: The Dirichlet problem at infinity}

\author[D. Li]{Dake Li}
\address[D. Li]{School of Mathematical Science\\
Fudan University\\
200433 Shanghai\\
P.R. China}
\email{dkli25@m.fudan.edu.cn}

\author[Z. Wang]{Zhizhang Wang}
\address[Z. Wang]{School of Mathematical Science\\
Fudan University\\
200433 Shanghai\\
P.R. China}
\email{zzwang@fudan.edu.cn}

\author[S. Yin]{Shiqi Yin}
\address[S. Yin]{School of Mathematical Science\\
Fudan University\\
200433 Shanghai\\
P.R. China}
\email{sqyin25@m.fudan.edu.cn}

%\thanks{*Corresponding author: Cong Wang.}
%----------classification, keywords, date
%\subjclass{35J60, 35J25, 35E10, 35B40}

\thanks{The second author is supported by NSFC Grant No.12141105.}

%\keywords{Minkowski space, Inverse curvature flow
% , Quermassintegral inequality}

\begin{abstract}
% In this paper, we study the Dirichlet problem at infinity for self-shrinkers of the inverse Hessian curvature flow in Minkowski space. For every continuous negative boundary value prescribed on the ideal boundary of hyperbolic space, we prove the existence and uniqueness of an entire spacelike strictly convex graphical hypersurface in Minkowski space.
% We further study the inverse Hessian curvature flow evolving from an entire spacelike strictly convex hypersurface satisfying a subsolution condition at infinity. After a rescaling, the flow exists for all time and converges locally smoothly to the unique self-shrinker with the prescribed boundary value at infinity.

This paper studies self-shrinkers and the long-time behavior of the inverse $\sigma_k$ curvature flow for noncompact entire spacelike strictly convex hypersurfaces in Minkowski space. In contrast to the co-compact setting, where the corresponding self-shrinker is rigid, the noncompact problem admits a rich family of self-shrinkers determined by their asymptotic data. More precisely, we formulate the self-shrinker equation as a fully nonlinear Dirichlet problem on hyperbolic space with prescribed data on its ideal boundary, and prove that every continuous negative boundary value determines a unique entire spacelike strictly convex self-shrinker. 
We further study the inverse $\sigma_k$ curvature flow starting from an entire spacelike strictly convex hypersurface satisfying a subsolution condition at infinity and a uniform positive lower bound for its $\sigma_k$ curvature. We prove global existence of the flow and show that the normalized flow converges locally smoothly to the unique self-shrinker with the same prescribed boundary value at infinity.
\end{abstract}

\maketitle

\section{Introduction}

% This paper concerns the inverse Hessian curvature flow in Minkowski space for noncompact hypersurfaces. The paper has two main parts. In the first part, we prove the existence of self-shrinkers with given boundary value at infinity for the inverse $\sigma_k$ curvature flow. In the second part, we show that, after rescaling, the flow converges to the self-shrinker, provided the initial hypersurface is chosen appropriately.

Let $\mathbb{H}^n$ be the hyperbolic space, that is the space form with constant sectional curvature $-1$. Let $g$ and $\nabla$ denote the standard metric and Levi-Civita connection of $\mathbb{H}^n$. $\Lambda \coloneqq \nabla^2 v - vg$ is the hyperbolic Hessian of a function $v$ defined on $\mathbb{H}^n$. Then we can define a $(1,1)$-tensor
\[
g^{-1} \circ \left( \nabla^2 v - vg \right) : T\mathbb{H}^n \to T\mathbb{H}^n. 
\]
Its eigenvalues are denoted by $\lambda \left[ \Lambda_v \right] = (\lambda_1, \cdots, \lambda_n)$. Let $1 \leqslant k \leqslant n$.
We consider the following Dirichlet problem
\begin{align}
  \begin{cases}
      \left( \frac{\sigma_n}{\sigma_{n-k}} \left( \lambda \left[ g^{-1} \circ \left( \nabla^2 v - vg \right) \right] \right) \right)^{\frac{1}{k}} = -\frac{v}{\binom{n}{k}^{\frac{1}{k}}} \quad \text{in } \mathbb{H}^{n}, \\
      v \rightarrow \varphi \quad \text{on } \mathbb{S}^{n-1}=\mathbb{H}^{n}(\infty), \label{sigmak selfshrinker}
  \end{cases}
\end{align}
where $\mathbb{H}^n(\infty)$ denotes the ideal boundary of $\mathbb{H}^n$ and $\sigma_k$ is the $k$-th elementary symmetric polynomial, i.e.
\[
\sigma_k(\lambda) = \sum_{1 \leqslant i_1 < \cdots < i_k \leqslant n} \lambda_{i_1}\cdots \lambda_{i_k}.
\]

It is clear that $v \equiv -1$ is a special solution of \eqref{sigmak selfshrinker}, which represents the hyperboloid. The boundary condition in \eqref{sigmak selfshrinker} is understood as the asymptotic sense at infinity, following \cite{Choi}. We will explain more in Section \ref{Preliminaries}. In this paper, we only consider the hyperbolic convex function $v$, namely that all eigenvalues of $\lambda[\Lambda_v]$ are positive. 

It follows from \eqref{sigmak selfshrinker} that any admissible
solution must be negative, thus we will restrict $\varphi \leqslant 0$. For convenience, since $\varphi \in C^{0}(\mathbb{H}^{n}(\infty))$ is defined on a compact set $\mathbb{H}^{n}(\infty)$, we assume that 
\begin{align}
    \varphi \leqslant \sup \varphi < 0. 
\end{align}

Equation \eqref{sigmak selfshrinker} arises from the study of the inverse $\sigma_k$ curvature flow in Minkowski space 
\begin{align}\label{eq:ICF}
    \frac{\partial X (p, t)}{\partial t} &= -\frac{1}{\sigma_k^{\frac{1}{k}}(\kappa\left[ \mathcal{M}_t \right])} \cdot \nu,
\end{align}
where $\kappa\left[ \mathcal{M}_t \right]$ are the principal curvatures of the hypersurface $\mathcal{M}_t$.
A special solution of the above equation is the self-shrinker satisfying  
\begin{align}\label{1.4}
    \sigma_k^\frac{1}{k}(\kappa\left[ \mathcal{M} \right]) = - \frac{\binom{n}{k}^{\frac{1}{k}}}{\langle X, \nu \rangle}.
\end{align}
It can be verified that the eigenvalues of hyperbolic Hessian $\Lambda$ of support function $v = \langle X, \nu \rangle$ are curvature radii of $\mathcal{M}$. Therefore, \eqref{1.4} can be rewritten as \eqref{sigmak selfshrinker}. Hence the geometric meaning of \eqref{sigmak selfshrinker} is the support function of hypersurfaces in Minkowski space.

In \cite{Ger} and \cite{Ur}, Gerhardt and Urbas respectively considered the inverse curvature flow in Euclidean space, which is a similar flow equation as \eqref{eq:ICF} with the opposite sign. Roughly speaking, they have proved that, for any compact initial hypersurface $\mathcal{M}_0$, the ICF will expand $\mathcal{M}_0$ and it will converge to some sphere by a suitable scaling. For entire graphical flows, Daskalopoulos-Huisken \cite{DH} studied the inverse mean curvature flow. 

In Minkowski space, Andrews, Chen, Fang and McCoy \cite{co-compact} first studied the curvature flow for cocompact hypersurfaces.  Aarons \cite{Aa}, Bayard-Schn\"{u}rer \cite{BayardSchnürer} and Bayard \cite{B09} studied the mean curvature flow, Gauss curvature flow and scalar curvature flow with a forcing term in Minkowski space. For the non-compact hypersurfaces, the second author and Xiao \cite{Entireself-expanders, Entiresigmakcurvatureflow} studied the Hessian curvature self-expanders and the Hessian curvature flows of entire spacelike hypersurfaces in Minkowski space. Then, Qu, the second author and Wo \cite{QWW24}  generalized a similar result to Hessian curvature translating solutions. Moreover, the second author and Xiao \cite{Entireconvexcurvatureflow} also studied some fully nonlinear curvature flows of noncompact spacelike hypersurfaces in Minkowski space. 

For prescribed curvature problem in Minkowski space, Treibergs \cite{Treibergs1982} and Choi-Treibergs \cite{ChoiTreibergs1990} studied the constant mean curvature equations. Then, Li \cite{Li1995}, Bayard-Schn\"urer \cite{BayardSchnürer} and Guan-Jian-Schoen \cite{GJS2006} considered the constant Gauss curvature equations. In 
\cite{Entireconstantsigmak,WX21}, the second author and Xiao considered the constant $\sigma_k$ curvature equations.  
In \cite{Theprescribedcurvatureproblem} Ren-Wang-Xiao also considered the prescribed $\sigma_k$ curvature equations, namely that, the right hand side is a given function, not constant.     

% A novelty of this paper is that we need the concept of boundary at infinity for negative curvature manifolds, which is first introduced by Eberlein and O'Neill in \cite{Visibilitymanifolds}. A remarkable application of this idea is due to Choi in \cite{Choi} and Anderson in \cite{Anderson}. Choi introduced the Dirichlet problem with boundary at infinity for harmonic equation. Anderson \cite{Anderson} and Sullivan \cite{sullivan1983dirichlet} solved this problem with additional curvature condition that sectional curvature pinched between two negative constants, thus showing that there exists a lot of bounded harmonic functions. 
% In \cite{AndersonSchoen}, a simplified proof is given by Anderson and Schoen, along with further results on the boundary at infinity.

The notion of the boundary at infinity for negatively curved manifolds was introduced by Eberlein and O'Neill \cite{Visibilitymanifolds} and has played an important role in the study of Dirichlet problems at infinity; see, for instance, Choi \cite{Choi}, Anderson \cite{Anderson}, Sullivan \cite{sullivan1983dirichlet}, and Anderson--Schoen \cite{AndersonSchoen}. In the present problem, the choice of asymptotic data is crucial. The usual asymptotic conditions on the graphical function or its Legendre transform are too coarse to distinguish different self-shrinkers. By contrast, the support function equation on $\mathbb{H}^n$ admits a comparison principle, and prescribing the support function on $\mathbb{H}^n(\infty)$ provides a natural asymptotic datum for the uniqueness problem. This motivates our formulation of the self-shrinker equation as a Dirichlet problem at infinity.

The first main result of the paper is

%\subsection{Main results}

\begin{theorem}\label{self-shrinker existence}
    Let $\varphi \in C^0(\mathbb{S}^{n-1})$ and $\varphi<0$. Then the Dirichlet problem \eqref{sigmak selfshrinker} admits a unique admissible solution $v \in C^\infty(\mathbb{H}^n) \cap C^0(\overline{\mathbb{H}}{}^n)$. Moreover, this solution is the support function of an entire spacelike strictly convex graphic hypersurface in Minkowski space.
\end{theorem}

\begin{remark}
In \cite{InverseHessianCurvatureFlow},  we have proved that, for co-compact hypersurfaces, the only self-shrinker is the hyperboloid. However, for non cocompact hypersurfaces here, we can find a lot of self-shrinkers.   
\end{remark}

According to \cite{Entireconstantsigmak}, the curvature equation \eqref{1.4} can be explicitly written as 
\begin{align}\label{1.5}
    \sigma_k^\frac{1}{k}\left( \frac{1}{w} \gamma^{ik} u_{kl} \gamma^{lj} \right) = - \frac{\binom{n}{k}^{1/k}}{\langle X, \nu \rangle},
\end{align}
where $u$ is the height function of hypersurface, $\gamma^{ij} = \delta_{ij} + \frac{u_i u_j}{w(1+w)}$ and $w = \sqrt{1 - |Du|^2}$. Let $u^*$ be the Legendre transform of $u$. Then the dual equation in $B_1$ model of \eqref{1.5} is
\begin{align}\label{1.6}
    F(w^* \gamma_{ik}^* u_{kl}^* \gamma_{lj}^*) = \left[ \frac{\sigma_n}{\sigma_{n-k}} (\kappa^* [w^* \gamma_{ik}^* u_{kl}^* \gamma_{lj}^*]) \right]^{\frac{1}{k}} = - \frac{u^*}{\binom{n}{k}^{1/k} w^*}, 
\end{align}
where $F=\left( \frac{\sigma_n}{\sigma_{n-k}} \right)^{\frac{1}{k}}$, $\gamma_{ij}^* = \delta_{ij} - \frac{\xi_i \xi_j}{1 + w^*}$ and $w^* = \sqrt{1 - |\xi|^2}$.  It is clear that \eqref{1.5} and \eqref{1.6}  do not satisfy the maximum principle. However, we observe that \eqref{sigmak selfshrinker} can satisfy the maximum principle. Therefore,  we propose problem \eqref{sigmak selfshrinker}, i.e. the Dirichlet problem with boundary at infinity on the hyperboloid. Note that the boundary condition used here is different from the previous works of the second author and Xiao \cite{Entireconstantsigmak,WX21,Entireself-expanders}. In fact, the boundary condition used there is  too rough for our equation in view of lack of maximum principle.  One may see that, every bounded function $\varphi$ in \eqref{sigmak selfshrinker} is corresponding to $u^*=0$ on $\partial B_1$ and $u\rightarrow |x|$ for $|x|\rightarrow +\infty$. Thus, it means that the boundary condition $u\rightarrow |x|+\phi(x/|x|)$ for $|x|\rightarrow+\infty$ and $u^*=-\phi$ on $\partial B_1$ cannot uniquely determine a hypersurface for self-shrinker equations.

Next, we consider the inverse $\sigma_k$ curvature flow. We first introduce the following condition. 

\begin{definition}\label{Cond}
Suppose $v_0$ is the support function of an entire strictly convex spacelike hypersurface $\mathcal{M}_{u_0}$. We say it satisfies \emph{subsolution condition at infinity} if there exists a compact set $K \subset \subset \mathbb{H}^n$
such that 
\begin{align}
    \sigma_k^{\frac{1}{k}}\left( \kappa(\mathcal{M}_{u_0}) \right) < - \frac{\binom{n}{k}^{\frac{1}{k}}}{\langle X, \nu \rangle} \quad \text{whenever } 
    \nu \in \mathbb{H}^n \setminus K,
\end{align}
or equivalently that
\begin{align}\label{strict subsolution}
    \left( \frac{\sigma_n}{\sigma_{n-k}} \right)^{\frac{1}{k}}(\Lambda_{v_0}) + \frac{v_0}{\binom{n}{k}^{\frac{1}{k}}} > 0 \quad \text{whenever } 
    \nu \in \mathbb{H}^n \setminus K.
\end{align}
%where the condition means when $\nu$ is sufficiently close to the ideal boundary of $\mathbb{H}^n$.
\end{definition}

%\begin{remark}\label{initial curvature condition}
   % From the subsolution condition at infinity, given $v_0$ is uniformly bounded, we have $\sigma_k\left( \kappa\left[ \mathcal{M}_{u_0} \right] \right) < C$ for some large constant $C$.
%\end{remark}

%In addition, we need to impose a technical condition
%\begin{align}\label{sigmak lower bound}
   % c_0  < \sigma_k\left( \kappa\left[ \mathcal{M}_{u_0} \right] \right)
%\end{align}
%for some positive constant $c_0$.

By using the above conditions, we can show that after rescaling, the rescaled inverse $\sigma_k$ curvature flow converges to the self-shrinker with given boundary value. 

\begin{theorem}\label{inverse Hessian curvature flow}
    Given an entire spacelike strictly convex hypersurface in Minkowski space $\mathcal{M}_{u_0}=\left\{ X = (x,u_0(x)) \in \mathbb{R}^{n,1} | x \in \mathbb{R}^n \right\}$. Suppose it satisfies the \emph{subsolution condition at infinity} and its support function has negative boundary value function  $\varphi \in C^0(\mathbb{S}^{n-1})$ on $\mathbb{H}^n(\infty)$.
    % If $k\neq n$, 
    We need further assume that there exists some small positive constant $c_0$ such that 
   \begin{align}\label{sigmak lower bound}
    c_0  \leqslant \sigma_k\left( \kappa\left[ \mathcal{M}_{u_0} \right] \right).
    \end{align}
    Then the inverse $\sigma_k$ curvature flow \eqref{eq:ICF} with the initial hypersurface $\mathcal{M}_{u_0}$ admits a unique solution for all time. Moreover, the rescaled flow 
        \[
        \tilde{X}(\cdot, t) = e^{\gamma t}X(\cdot, t), \quad \gamma = \frac{1}{\binom{n}{k}^{\frac{1}{k}}}.
        \]
    converges locally smoothly to the unique self-shrinker with boundary value $\varphi$ at infinity.
\end{theorem}
\begin{remark}\label{sigmak upper bound}
    Indeed, the support function of $\mathcal{M}_{u_0}$ is bounded between two negative constants by continuity. Therefore, by the subsolution condition at infinity, we have 
    $$0<\sigma_k\left( \kappa\left[ \mathcal{M}_{u_0} \right] \right) \leqslant C$$ for some large constant $C$.
\end{remark}

\begin{remark} The assumptions  in  Theorem  \ref{inverse Hessian curvature flow} are fairly mild. In section 13, we will give a large collection of initial hypersurfaces to satisfy conditions in Theorem  \ref{inverse Hessian curvature flow}. Roughly speaking, on sufficiently large compact subsets in the interior of the initial hypersurface, it can be prescribed arbitrarily.
\end{remark}
%The subsolution condition at infinity is an improvement of the condition in \cite{Entiresigmakcurvatureflow}. 
Finally, we provide a counterexample showing that the long-time existence result in Theorem \ref{inverse Hessian curvature flow} and the existence of smooth solution in Theorem \ref{self-shrinker existence} cannot, in general, be extended to the inverse quotient curvature flows and quotient self-shrinkers.

\begin{theorem}\label{Counterexample}
(1) There exists a family of smooth entire spacelike strictly convex hypersurfaces $\mathcal{M}_t$ for finite time $0<t<T$ satisfying 
the inverse quotient curvature flow
\begin{align}\label{IQCF}
    \frac{\partial X}{\partial t}
    =
    -\frac{\sigma_{n-1}}{\sigma_n}(\kappa[\mathcal{M}_t])\,\nu,
\end{align}
and as $t\rightarrow T$, $\mathcal{M}_t$ develops a singularity in interior.

%More precisely, there exist suitable small $\varepsilon>0$ and a smooth function $\hat{v}$ on $\mathbb H^n$, such that
%\[
%v(y,t) = -e^{-nt}+\varepsilon \hat{v}(y),
%\quad y\in\mathbb H^n,
%\]
%is a strictly convex solution of the flow equation, for $0<t<T$,
%\[
%\frac{\partial v}{\partial t}
%=
%\sigma_1(\Lambda[v]),
%\quad
%\Lambda[v]=\nabla^2v-v\bar g.
%\]
%Moreover, $\Lambda[v]$ is positive definite for $t=0$ and it will become degenerate at $T$. Consequently, the corresponding spacelike strictly convex hypersurface will develop singularity (i.e. some principal curvature is infinite) at a finite time.

(2) If $n \geqslant 3$, consider self-shrinker equation to the flow \eqref{IQCF}, 
\begin{align}\label{QS}
    \frac{\sigma_{n-1}}{\sigma_n}\left[ \kappa \right] = - n \langle X, \nu \rangle.
\end{align}
We can give an entire spacelike convex hypersurface such that its principal curvature only has singularity in interior, namely that it is smooth near infinity. Moreover, its support function is bounded on ideal boundary $\mathbb{H}(\infty)$.      
\end{theorem}

Note that we still cannot extend the above examples to the general inverse quotient curvature flows.

This paper is organized as follows. In Section 2, we review the concepts of boundary at infinity, distance function to convex sets and the comparison principle. Sections 3-6 solve the existence of \eqref{sigmak selfshrinker}, namely the proof of Theorem  \ref{self-shrinker existence}.
In detail, Section 3 will construct the upper and lower barriers for \eqref{sigmak selfshrinker}. Section 4 solves the approximate problem of \eqref{sigmak selfshrinker} on every geodesic ball of $\mathbb{H}^n$. In Section 5, we first consider the local estimate and existence of the self-shrinkers for the inverse Gauss curvature flow. Section 6 gives the local estimates and existence of   the self-shrinkers for the general inverse $\sigma_k$ curvature flow. Sections 7-12 prove Theorem \ref{inverse Hessian curvature flow}. In Section 7, we rewrite the inverse curvature flow \eqref{eq:ICF} by the support function, namely equation \eqref{support function flow}. Section 8 constructs the upper and lower barriers for flow equations. Section 9 solves the approximate problem for flow equations. Section 10 and Section 11 prove the existence of entire solution for the inverse Hessian curvature flow by using the local estimates. Section 12 considers the convergence of the rescaled flow. Section 13 discusses the existence of the initial hypersurface. The last Section gives two counterexamples, i.e. Theorem \ref{Counterexample}.

\section{Preliminaries}\label{Preliminaries}

\subsection{Boundary at infinity}\label{Boundary at infinity}

A \textit{Cartan--Hadamard manifold} $M$ is a complete simply connected Riemannian manifold with sectional curvature $\leqslant 0$. We can make a compactification of $M$, denoted by $\overline{M}$, by attaching the boundary at infinity $M(\infty)$.

For readers' convenience, we give a brief introduction here for some properties of Cartan--Hadamard manifolds. One may see   \cite{Visibilitymanifolds, Choi,BishopO'Neill} and references therein for more detail. 

\begin{definition}
    On a Cartan--Hadamard manifold $M$, two unit-speed geodesics $\gamma_i : \mathbb{R} \rightarrow M \, (i=1,2)$ are said to be \textit{asymptotic} if there exists a constant $C$ such that $d(\gamma_1 (t), \gamma_2 (t)) \leqslant C , \, \forall t \geqslant 0$, where $d$ is the distance function of $M$.
\end{definition}

\begin{lemma}
    The asymptotic relation for unit-speed geodesics in $M$ is an equivalence relation.
\end{lemma}

\begin{lemma}[Uniqueness]
    Asymptotic class satisfies the following uniqueness:

    \textup{(1)} If two unit speed asymptotic geodesics have a point in common, then they are the same. 

    \textup{(2)} Given a geodesic $\alpha$ and a point $p \in M$, there exists a unique geodesic $\beta$ such that $\beta(0) = p$ and $\beta$ is asymptotic to $\alpha$.
\end{lemma}

\begin{definition}
    The asymptotic class of $\gamma$ is denoted by $\gamma (\infty)$. Similarly, $\gamma (-\infty)$ denotes the asymptotic class of the reverse curve of $\gamma$.
\end{definition}

\begin{definition}
    The \textit{boundary at infinity} or the \textit{Eberlein-O'Neill boundary} \cite{Visibilitymanifolds} of $M$ is the set $M(\infty)$ of all asymptotic classes of geodesics in $M$. Define $\overline{M} = M \sqcup M(\infty)$.
\end{definition}

\begin{proposition}
    Let $p$ be any fixed point of $M$. Then $M(\infty)$ and $S_p$ are in one-to-one correspondence, where $S_p$ is the set of unit tangent vectors at $p$, $S_p = \{v \in T_p M \mid \langle v, v \rangle = 1\}$. Namely, if we denote $\gamma_v$ the geodesic starting from $p$ with initial velocity $v$, then
    \[
    M(\infty) = \left\{ \gamma_v(\infty) | v \in S_p \right\}.
    \]
\end{proposition}

% $\forall p,q \in M$, let $\gamma_{pq}$ denote the unique geodesic from $p$ to $q$.
Generally, $\forall x,y \in \overline{M}$, let $\gamma_{xy}$ denote the unique geodesic through $x$ and $y$ if exists. When $x$ or $y \in M(\infty)$, it means that $\gamma (-\infty) = x$ or $\gamma (\infty) = y$. If $x$ and $y$ are both in $M(\infty)$, a stronger curvature condition is needed to ensure existence and uniqueness of $\gamma_{xy}$. A sufficient condition is that sectional curvature of $M$ is less than a constant $-c, c>0$, the proof can be found in \cite{Visibilitymanifolds}.

Let $p$ be a point of $M$. For $v, w \in T_p{M}$, define $\angle_p(v, w)$ to be the angle between $v$ and $w$ in the vector space $T_p{M}$. If $x, y \in M$, then we define the angle $\angle_p(x, y) = \angle_p(\gamma^{\prime}_{px}(0), \gamma^{\prime}_{py}(0))$.

For a fixed point $p \in M$, we have cone with vertex at $p$, opening angle $\delta$ and axis $v \in T_p{M}$ as 
\[
C(v,\delta) = \{x \in \overline{M} \mid  \angle_{p}(v, \gamma_{px}^{\prime}(0)) < \delta\}.
\]
Let the truncated cone with $r>0$ be
\[
T(v, \delta, r) = C(v, \delta) - \{q \in M \mid d(p,q) \leqslant r\}.
\]
It can be verified that the family $\mathscr{B} = \{T(v, \delta, r), B_{r}(q) \mid \forall v \in T_p{M}, \delta>0, r>0, q\in M \}$ forms a basis of a topology, which is called the \textit{cone topology}.

\begin{remark}
    The topology on $M$ induced by cone topology is the original topology of $M$, and $M$ is a dense open set of $\overline{M}$.
\end{remark}

\begin{definition}[Convergence in cone topology]\label{convergence in cone topology}
Given $v \in S_p$, the truncated cone in $M$, $K(v, \delta, r)$ is defined to be $T(v, \delta, r) \cap M = \{ q \in M \mid d(p, q) > r \text{ and } \angle_{p}(v, \gamma_{pq}^{\prime}(0)) < \delta \}$. Suppose $u$ is a function defined only on $M$. We say $u$ converges to a number $A$ as $q \rightarrow \gamma_v(\infty)$ if given $\varepsilon > 0$, there exists some $\delta > 0$ and $r > 0$ such that $|u(q) - A| \leqslant \varepsilon$ for all $q \in K(v, \delta, r)$.
\end{definition}

In this paper, what we consider is just hyperbolic space $\mathbb{H}^n$ and its natural boundary $\mathbb{H}^n(\infty) = \mathbb{S}^{n-1}$. Boundary value condition of \eqref{sigmak selfshrinker} is explained in the sense of Definition \ref{convergence in cone topology}. Indeed, if we consider Klein ball model $B_1$ for hyperbolic space, the compactification at infinity is actually the closed unit ball, with the topology induced by the Euclidean metric.

\subsection{Distance function to some convex set}

Motivated by \cite{Choi}, we can use distance function to construct barriers. For convenience, we state Lemma 4.3 in \cite{Choi} as follows.
\begin{lemma}[Lemma 4.3 in \cite{Choi}]\label{Hessian of distance function}
    Suppose $M$ is a $n$-dimensional Cartan--Hadamard manifold with sectional curvature $\leqslant -1$. If $s : M \rightarrow \mathbb{R}$ is the distance function from some convex set in $M$, then $\Delta s \geqslant (n-1) \tanh s$. Moreover we have
    \[
    \begin{aligned}
            \operatorname{Hess} s \geqslant \tanh s \, (g - \mathrm{d}s \otimes \mathrm{d}s). 
    \end{aligned}
    \]
\end{lemma}

\subsection{Comparison principle}

Recall equation
\begin{align}
  \begin{cases}
      \left( \frac{\sigma_n}{\sigma_{n-k}} \left[v_{ij}-v\delta_{ij}\right] \right)^{\frac{1}{k}} = -\frac{v}{\binom{n}{k}^{\frac{1}{k}}} \quad \text{in } \mathbb{H}^{n}, \\
      v \rightarrow \varphi \quad \text{on } \mathbb{S}^{n-1}=\mathbb{H}^{n}(\infty), \label{sigmak selfshrinker2}
  \end{cases}
\end{align}
where $\varphi \in C^0(\mathbb{S}^{n-1})$ and $\varphi<0$.

% Recall dual equation of equation \eqref{quotient}
% \begin{align}
%         \frac{1}{H} = -\frac{1}{n} \langle X, \nu \rangle \quad \text{in } \mathcal{M} \label{selfshrinker}
% \end{align}

Define the admissible set of equation \eqref{sigmak selfshrinker2}:
\begin{align}\label{admissible set}
    \mathcal{S} \coloneqq \left\{ v < 0, v \in C^{2} \mid \lambda \left[ g^{-1} \circ (\nabla^2 v - v g) \right] \in \Gamma_n \right\}.
\end{align}

\begin{lemma}[Comparison principle]\label{Comparison principle}
    Suppose $F=\left( \frac{\sigma_n}{\sigma_{n-k}} \right)^{\frac{1}{k}}$, then let $F(w)=F(w_{ij}-w\delta_{ij})$ by a slight abuse of notation. For $w,v \in \mathcal{S}$, if $F(w) + \frac{w}{\binom{n}{k}^{\frac{1}{k}}} \geqslant F(v) + \frac{v}{\binom{n}{k}^{\frac{1}{k}}}$ and $w \leqslant v$ on $\mathbb{S}^{n-1}$ in cone topology sense, we have $w \leqslant v$ in $\mathbb{H}^{n}$. 
\end{lemma}
\begin{proof}
    Denote $G(w) \coloneqq F(w) + \frac{w}{\binom{n}{k}^{\frac{1}{k}}}$, then 
    \begin{align*}
        G(w) - G(v)= a_{ij}(w-v)_{ij} - c (w-v) \geqslant 0.
    \end{align*}
    where 
    \begin{align*}
        a_{ij} &= \int_{0}^{1} F^{ij}\left( tw + (1-t)v \right) dt, \\
        c &= \int_{0}^{1} \left( \sum_{i}F^{ii}\left( tw + (1-t)v \right) - \frac{1}{\binom{n}{k}^{\frac{1}{k}}} \right) dt \geqslant 0,
    \end{align*}
    by homogeneity and concavity of $F$.

    Suppose there exists a point $y \in \mathbb{H}^n$, such that $w(y) - v(y) = 2\varepsilon > 0$. Since $w \leqslant v$ on $\mathbb{S}^{n-1}$, for any direction $\ell \in \mathbb{S}^{n-1}$, there exist some $\delta > 0$ and $r > 0$ such that $(w-v)(q) \leqslant \varepsilon$ for all $q \in K(\ell, \delta, r)$. By compactness we can choose finite $\{\ell_i, \delta_i,r_i\}_{i=1}^k$ such that $K(\ell_i, \delta_i, r_i)$ cover all directions. Let $R = \max\{r_i\}$, then $(w-v) \leqslant \varepsilon$ outside $B_R$, i.e. $y \in B_R$ and $\left.(w-v)\right|_{\partial B_R} \leqslant \varepsilon$.
    By maximum principle for linear elliptic equation, we have
    \begin{align*}
        w \leqslant v + \varepsilon \quad \text{in } B_R.
    \end{align*}
    We have a contradiction. 
\end{proof}

It is well known that the hyperbolic space $\mathbb{H}^n(-1)$ can be  canonically embedded into $\mathbb{R}^{n,1}$ as the hyperboloid
\[
\mathbb{H}^n(-1) \coloneqq \{X \in \mathbb{R}^{n,1} \mid \langle X, X \rangle = -1, x_{n+1}>0 \}.
\]
Let $B_1^{\xi}$ be the unit ball in hyperplane $\{x_{n+1}=1\}$. The Beltrami map $P$ is the projection centered at the origin: 
\begin{align*}
    P: \mathbb{H}^n &\longrightarrow B_1^{\xi} \\
    (y_1, \cdots, y_{n+1}) &\longmapsto (\xi_1, \cdots, \xi_n),
\end{align*}
where $y_{n+1} = \sqrt{1+y_1^2+\cdots+y_n^2}$ and $\xi_i = \frac{y_i}{y_{n+1}}$.

We now extend the concept of hyperbolic convexity to continuous functions. We say a continuous function $v$ defined in $\mathbb{H}^n$ is \emph{hyperbolic convex} if $v\left( y(\xi) \right) \sqrt{1-|\xi|^2}$ is convex in $B_1^{\xi}$, where $y(\xi) = P^{-1}(\xi)$.

From now on, we always denote $F(v)=\left( \frac{\sigma_n}{\sigma_{n-k}} \right)^{\frac{1}{k}}(\Lambda_v)$ and $G(v) =F(v) + \frac{v}{\binom{n}{k}^{\frac{1}{k}}}$.

\begin{definition}[Viscosity subsolution]
    We say a hyperbolic convex function $v \in C^0(\mathbb{H}^n)$ is a viscosity subsolution to equation
    \begin{align}
        F(h_{ij}-h\delta_{ij}) = -\frac{h}{\binom{n}{k}^{\frac{1}{k}}}, \label{viscosity subsolution equation}
    \end{align}
    if for any point $y \in \mathbb{H}^n$, and every $\varphi \in C^2$ defined in a neighborhood of $y$, such that $v-\varphi$ has a local maximum at $y$ and
    \[
    v(y) = \varphi(y),
    \]
    we have
    \begin{align*}
        F(\varphi_{ij}-\varphi\delta_{ij}) \geqslant -\frac{\varphi}{\binom{n}{k}^{\frac{1}{k}}}
    \end{align*}
    at $y$.
\end{definition}

% \begin{remark}
%     Viscosity supersolution are defined by same way. Viscosity solution are both viscosity subsolution and viscosity supersolution.
% \end{remark}

\begin{lemma}[Viscosity comparison principle]\label{Viscosity comparison principle}
    Suppose $w$ is a viscosity subsolution of \eqref{viscosity subsolution equation}.
    Suppose $v \in \mathcal{S}$ is a $C^2$ supersolution of \eqref{viscosity subsolution equation}.
    If $w \leqslant v$ on $\mathbb{S}^{n-1}$ in cone topology sense, then $w \leqslant v$ in $\mathbb{H}^{n}$.
\end{lemma}
\begin{proof}
    We first claim that the viscosity comparison principle holds for any geodesic ball $B_R$. Suppose that $w$ and $v$ are viscosity subsolution and admissible $C^2$ supersolution in $B_R$ respectively, with
    \begin{align}
        w \leqslant v \quad \text{on } \partial B_R. 
    \end{align}
    We will prove that
    \begin{align}
        w \leqslant v \quad \text{in } B_R. \label{BR leq}
    \end{align}
    Let $\rho$ be the geodesic distance function from the origin. Define
    \[
    g(\rho) = - \left( \cosh \rho \right)^{-a}
    \]
    as in Lemma \ref{smooth function g}, then $g$ is hyperbolic convex, and $g$ is bounded. By Lemma \ref{smooth function g}, we can see that    
    \[
    G(g)=F(g) + \frac{g}{\binom{n}{k}^{\frac{1}{k}}}\geqslant c_0 \quad \text{in } B_R,
    \]
    for some positive constant $c_0$. 
 
    Consider 
    \[
    w_\eta = w + \eta g.
    \]
    We claim that $w_\eta$ is a viscosity subsolution to equation
    \begin{align}\label{Modified subsolution}
        G(h) = \eta c_0.
    \end{align}
    For any $x_0 \in B_R$, suppose  $\varphi_{\eta} \in C^2$ satisfies 
    \[
    \varphi_{\eta}(x_0) = w_{\eta}(x_0), \quad \varphi_{\eta}(z) \geqslant w_{\eta}(z), \quad \mathbb\forall z \in B_R.
    \]
    Let $\varphi = \varphi_{\eta} - \eta g$, then
    \[
    \varphi(x_0) = w(x_0), \quad \varphi(z) \geqslant w(z), \quad \mathbb\forall z \in B_R.
    \]
    Since $w$ is a viscosity subsolution in $B_R$, we have
    \[
    G(\varphi)(x_0) \geqslant 0.
    \]
    By concavity we know at $x_0$
    \begin{align*}
        G(\varphi_{\eta}) \geqslant G(\varphi) + G(\eta g) \geqslant \eta c_0.
    \end{align*}
    By the hyperbolic convexity of $w$, we know that in $B_1^\xi$ model, $w \sqrt{1-|\xi|^2}$ is convex, which implies 
    $(w +\eta g)\sqrt{1-|\xi|^2}$ is also convex. Thus, $w_{\eta}$ is a viscosity subsolution.

    Now suppose \eqref{BR leq} fails. We may assume  that there exists $\eta>0$ such that 
    \[
    \max_{\overline{B}_R} \left( w_{\eta}-v \right) > \max_{\partial B_R} \left( w_{\eta}-v \right).
    \]
    Suppose at $z_0 \in B_R$, we have
    \[
    w_{\eta}(z_0)-v(z_0) = M = \max_{\overline{B}_R} \left( w_{\eta}-v \right).
    \]
    Thus we have
    \[
    v(z_0) + M = w_{\eta}(z_0), \quad v + M \geqslant w_{\eta} \quad \text{in } \overline{B}_R.
    \]
    Therefore since $w_{\eta}$ is a viscosity subsolution to \eqref{Modified subsolution}, at $z_0$, we have
    \[
    G(v+M) \geqslant \eta c_0,
    \]
    which shows that at $z_0$
    \begin{align*}
        G(v) = G((v+M) + (-M)) \geqslant G(v+M) + G(-M) \geqslant \eta c_0.
    \end{align*}
    This contradicts $G(v) \leqslant 0$ since $v$ is a supersolution.

    Now we can prove the viscosity comparison principle. Suppose not, then there exists a point $y_0 \in \mathbb{H}^n$, such that $w(y_0) - v(y_0) = 2\varepsilon > 0$. By the same argument as in Lemma \ref{Comparison principle}, we may still assume that $(w-v) \leqslant \varepsilon$ outside $B_R$. Then in $\overline{B}_R$, we have 
    \begin{align*}
        w - \varepsilon \leqslant v \quad \text{on } \partial B_R,
    \end{align*}
    and $w - \varepsilon$ is a viscosity subsolution in $\overline{B}_R$. Using the viscosity comparison principle for geodesic ball $B_R$, we have
    \[
    w - \varepsilon \leqslant v \quad \text{in } B_R.
    \]
    Contradiction. The viscosity comparison principle is proved.
\end{proof}

\begin{lemma}
    For a family of viscosity subsolution $\{v_{\alpha}\}_{\alpha \in \mathcal{L}}$, their supremum $v \coloneqq \sup_{\alpha \in \mathcal{L}} v_{\alpha}$ is still a viscosity subsolution.
\end{lemma}

\section{The construction of barriers: self-shrinker}\label{The construction of barriers: self-shrinker}

\subsection{Construction of upper barrier function}

Here we want to construct $\overline{v}$ defined on $\mathbb{H}^n$, which is not a supersolution but an upper barrier of \eqref{sigmak selfshrinker}, since $\overline{v}$ is not hyperbolic convex in general. 

Let $\overline{v}$ be the harmonic function with $\overline{v} = \varphi$ on $\mathbb{H}^{n}(\infty)$, i.e.
\[
\begin{cases}
    \Delta \overline{v} = 0 \quad \text{in} \quad \mathbb{H}^{n}, \\
    \overline{v} = \varphi \quad \text{on} \quad \mathbb{S}^{n-1}. 
\end{cases}
\]
By using a special case of the main theorem in \cite{Anderson}, the solution exists and is unique. By asymptotic maximum principle (Proposition 2.5 in \cite{Choi}), we have $\inf \varphi \leqslant \overline{v} \leqslant \sup \varphi$. 

If $v$ solves \eqref{sigmak selfshrinker}, $v$ is hyperbolic convex, then by generalized Newton--Maclaurin's inequality, we have
\begin{align*}
    \left( \frac{\sigma_n}{\sigma_{n-k}} \right)^{\frac{1}{k}} \left( v_{ij}-v\delta_{ij} \right) \leqslant \frac{1}{\binom{n}{k}^{\frac{1}{k}}} \frac{\sigma_1}{n} \left( v_{ij}-v\delta_{ij} \right)
    = \frac{1}{\binom{n}{k}^{\frac{1}{k}}} \left( \frac{\Delta v}{n} - v \right) \quad \Rightarrow \Delta v \geqslant 0.
\end{align*}
By asymptotic maximum principle for subharmonic function, $v \leqslant \overline{v}$.

\subsection{Construction of lower barrier function}

Similarly by generalized Newton--Maclaurin's inequality, we have
\begin{align}
    \left[ \frac{\frac{\sigma_n (\lambda)}{\binom{n}{n}}}{\frac{\sigma_{n-1} (\lambda)}{\binom{n}{n-1}}} \right]^{1} \leqslant \left[ \frac{\frac{\sigma_n (\lambda)}{\binom{n}{n}}}{\frac{\sigma_{n-k} (\lambda)}{\binom{n}{k}}} \right]^{\frac{1}{k}}
    \leqslant \left[ \frac{\frac{\sigma_n (\lambda)}{\binom{n}{n}}}{\frac{\sigma_0 (\lambda)}{\binom{n}{0}}} \right]^{\frac{1}{n}}.
\end{align}
Thus in order to verify some function is a subsolution of \eqref{sigmak selfshrinker}, we only need to check it is a subsolution of \eqref{sigmak selfshrinker} when $k=1$. In this section, we restrict $F=\frac{\sigma_n}{\sigma_{n-1}}$. We will construct a subsolution to the following equation
\begin{align}
    \begin{cases}
        \frac{\sigma_n}{\sigma_{n-1}} \left[v_{ij}-v\delta_{ij}\right] = -\frac{1}{n} v \quad \text{in } \mathbb{H}^{n}, \\
        v \rightarrow \varphi \quad \text{on } \mathbb{S}^{n-1} = \mathbb{H}^{n}(\infty). \label{k=1 selfshrinker}
    \end{cases}
\end{align}

\begin{lemma}\label{f(s)}
    Let $s : \mathbb{H}^n \setminus K \rightarrow \mathbb{R}^{+}$ be the distance function from some convex set $K \subseteq \mathbb{H}^n$. Then there exists a function $f: \mathbb{R}^{+} \to \mathbb{R}^{-}$ such that $f(s) \to 0$ as $s \to +\infty$, $f(s)$ is an admissible function, i.e. $f \in \mathcal{S}$, see \eqref{admissible set}, and 
    \[
    \left[ F(f) + \frac{f}{n} \right](s) \geqslant 0
    \]
    for all $s \geqslant s_0$, where $s_0$ satisfies $\tanh s_0 = \frac{1}{2}$.
\end{lemma}
\begin{proof}
For convenience, we assume that $f: \mathbb{R}^{+} \to \mathbb{R}^{-}$ is a smooth function with $f^{\prime}(s)>0, f^{\prime \prime}(s)<0$, then by Lemma \ref{Hessian of distance function} we have
\begin{align*}
    \text{Hess}(f(s)) &= f^{\prime}(s) \text{Hess}(s) + f^{\prime \prime}(s) \mathrm{d}s \otimes \mathrm{d}s \\
    &\geqslant         
    f^{\prime}(s) (\tanh s) g - \left( f^{\prime}(s)\tanh s - f^{\prime \prime}(s) \right) \mathrm{d}s \otimes \mathrm{d}s. 
\end{align*}
The tensor in the last line has eigenvalues of $(n-1) \cdot \left[ f^{\prime}(s) \tanh s \right]$ and $1 \cdot \left[ f^{\prime \prime}(s) \right]$. To ensure that $f(s)$ to be hyperbolic convex, We also need
\[
-f(s) + f^{\prime \prime}(s) > 0.
\]
For subsolution condition, it demands
\[
F(f) + \frac{f}{n} \geqslant 0.
\]
A direct calculation gives
\begin{align}\label{subsolution condition}
    F + \frac{f}{n} &= \frac{1}{\sigma_{1}\left( \Lambda^{-1} \right)} + \frac{f}{n} 
    \geqslant \frac{1}{\frac{n-1}{f^{\prime} \tanh s-f} + \frac{1}{f^{\prime \prime}-f}} + \frac{f}{n} \notag \\
    &= \frac{\left( f^{\prime} \tanh s-f \right)\left( f^{\prime \prime}-f \right)}{\left( n-1 \right)\left( f^{\prime \prime}-f \right) + f^{\prime} \tanh s-f} + \frac{f}{n} \notag \\
    &= (-f) \left[ \frac{\left( 1-\frac{f^{\prime} \tanh s}{f} \right) \left( 1-\frac{f^{\prime \prime}}{f} \right)}{\left( n-1 \right)\left( 1-\frac{f^{\prime \prime}}{f} \right) -\frac{f^{\prime}}{f} \tanh s +1} - \frac{1}{n} \right].
\end{align}
Combining the above conditions, what we want is 
\[
\begin{cases}
    f: [0, \infty) \to \mathbb{R^{-}}, \\
    f<0, f^{\prime}>0, f^{\prime \prime}<0, \\
    \lim_{s \to +\infty} f(s) \to 0,\\
    -f + f^{\prime \prime} > 0, \\
    \left[ \frac{\left( 1-\frac{f^{\prime} \tanh s}{f} \right) \left( 1-\frac{f^{\prime \prime}}{f} \right)}{\left( n-1 \right)\left( 1-\frac{f^{\prime \prime}}{f} \right) -\frac{f^{\prime}}{f} \tanh s +1} - \frac{1}{n} \right] \geqslant 0.
\end{cases}
\]
We also require that $f$ vanishes at infinity because of the construction of subsolution later in this section. Now let 
\[
f(s) = -e^{ks},
\]
where $k$ is a constant to be determined later.
Then we get
\begin{align*}
    f^{\prime} = -ke^{ks} > 0 &\Longleftrightarrow k < 0, \\
    f^{\prime \prime} = -k^{2}e^{ks} < 0 &\Longleftrightarrow k \in \mathbb{R}\setminus\{0\}, \\
    -f + f^{\prime \prime} > 0 &\Longleftrightarrow |k| < 1, \\
    \lim_{s \to +\infty} f(s) \to 0 &\Longleftrightarrow k<0. 
\end{align*}
Hence we restrict $-1 < k < 0$. For the last condition, after finding a common denominator, the numerator is 
\begin{align}\label{numerator}
    &\left( 1-k\tanh s \right) \left( 1-k^{2} \right) - \frac{n-1}{n} \left( 1-k^{2} \right) + \frac{k}{n} \tanh s - \frac{1}{n} \notag \\
    =&-\frac{k^2}{n} - k\tanh s \left( 1-k^{2} \right) + \frac{k}{n} \tanh s \notag \\
    \geqslant&\left( \frac{3}{4} - \frac{1}{n} \right) \left( -k\tanh s \right) -\frac{k^2}{n} \notag \\
    \geqslant&\frac{1}{4} \left( -k\tanh s \right) -\frac{k^2}{2} \notag \\
    \geqslant&\frac{-k}{8} -\frac{k^2}{2} = \frac{1}{8}\left( -k - 4k^2 \right) \notag \\
    =& \frac{1}{8}(-k)(1+4k ) \geqslant 0, 
\end{align}
given $-\frac{1}{4} \leqslant k <0$, and $s$ is sufficiently large such that $\tanh s \geqslant \frac{1}{2}$. Then, we fix some $-\frac{1}{4} \leqslant k <0$ and $f(s) = -e^{ks}$, the lemma is now proved.
\end{proof}

From now on, in order to obtain the boundary value, we need to consider the summation of a family of function $\{f(s_i)\}_{i=1}^k$, where each $s_i$ is the distance function from a convex set $K_i$. By the concavity and homogeneity with degree one of the operator $F$, we know $h = \sum_{i=1}^{k} f(s_i)$ also satisfies
\[
F(h) + \frac{h}{n} \geqslant 0
\]
in the common definition domain of $s_i$ where $s_i \geqslant s_0$ for any $1 \leqslant i \leqslant k$.

\begin{lemma}
    There exists a viscosity subsolution $\underline{v}$ to equation \eqref{k=1 selfshrinker} defined in $\mathbb{H}^n$ with boundary value $\varphi$ at infinity.
\end{lemma}
\begin{proof}
We give a constructive proof directly and divide into three steps. The construction is adapted from Theorem 4.7 in \cite{Choi}.

\textbf{Step 1:} Fix some origin $O \in \mathbb{H}^n$. Then for any direction $v\in \mathbb{S}^{n-1}$, given $\varepsilon > 0$, there exists a sufficiently small $\delta(\varepsilon)>0$ depending on the modulus of continuity of $\varphi$, such that
\[
|\varphi(w) - \varphi(v)| \leqslant \varepsilon \quad \text{for } \measuredangle (w, v) \leqslant \delta, w \in \mathbb{S}^{n-1}.
\]
Let $S(v, \delta) = \{ w \in \mathbb{S}^{n-1} | \measuredangle (w, v) = \delta \}$, then $S(v, \delta)$ is compact. By Theorem 3.2 in \cite{Anderson}, we know that for any $x, y \in \mathbb{S}^{n-1}$, $x \neq y$, there exist disjoint open neighborhoods $V_x, V_y$ in $\overline{\mathbb{H}}{}^n$ relative to the cone topology such that $V_x \cap \mathbb{H}^n$ and $V_y \cap \mathbb{H}^n$ are convex. Fix one point $v$, let $w$ vary in $S(v, \delta)$, we obtain disjoint open neighborhoods $V_{w}$ and $V_v(w)$. By compactness of $S(v, \delta)$, there exists a finite family of points $w_i \in S(v, \delta)$, $i=1,2,\cdots,k$, such that $\{V_{w_i}\}_{i=1}^k$ cover $S(v, \delta)$ and 
\[
V_{w_i} \cap V_{v} = \varnothing \quad \forall \, 1 \leqslant i \leqslant k,
\]
where $V_{v} = \cap V_{v}(w_i)$ is an open neighborhood of $v$.

\textbf{Step 2:} Let $\Omega = \mathbb{H}^n \setminus \cup_{i=1}^k V_{w_i}$ and $s_i: \Omega \to \mathbb{R}^{+}$ the distance function from $V_{w_i}$. Then define 
\[
B(x) \coloneqq 
\sum_{i=1}^{k} f(s_i(x)).
\]
It should be noted that this function is not defined on the whole space. To overcome this obstacle, we need to do some cutoff. Consider
\[
R \coloneqq \sup \{\operatorname{dist}(O, \cup_{i=1}^k V_{w_i} \cap \gamma_{w}) | w \in S(v, \delta) \},
\]
where $\operatorname{dist}$ denotes the distance between two sets, and $\gamma_{w}$ denotes the geodesic curve with tangent vector $w$ at origin. Since $S(v, \delta)$ is compact, we know $R$ has a uniform upper bound. Then we can see that $\Omega \setminus B_{2R}$ are divided into two components. Denote the one that has nonempty intersection with $V_{v}$ by $C_1$, and the other one by $C_2$. Let
\[
L = \{ x| s_i(x) \geqslant s_0,  \forall 1 \leqslant i \leqslant k \},
\]
where $s_0$ satisfies $\tanh s_0 = \frac{1}{2}$.
Then let $\Omega_1 = C_1 \cap L$, and $\overline{\Omega}_2 = \{x \in \Omega_1 | B(x) \geqslant c\}$, where $c$ is a negative constant, sufficiently close to $0$, and to be determined later. Specifically, we can see on $\partial \Omega_1$, $B$ is uniformly bounded from above by a negative constant. Indeed
\[
B(x) \leqslant \max\{f(4R), f(s_0), -1\} \eqqcolon d < 0  \quad \text{on } \partial \Omega_1.
\]
Then choose $c > d$, we have $\overline{\Omega}_2$ is away from $\Omega_1$ around any finite point.
% $\overline{\Omega}_2 \subset \subset \Omega_1$. 
Thus we can define
\[
B_{v, \delta}(x) =
\begin{cases}
    B(x), &x \in \overline{\Omega}_2 \\
    c, &x \in \mathbb{H}^n \setminus \overline{\Omega}_2.
\end{cases}
\]
$B_{v, \delta}$ is defined on $\mathbb{H}^n$ depending on $v$ and $\delta$ and satisfies the following properties:
\begin{enumerate}[label=\text{(\arabic*)}]
    \item $B_{v,\delta} \leqslant 0$ and $\lim B_{v,\delta}(q) = 0$ as $q \to v$;
    \item $\exists \, \eta > 0$ such that $\limsup B_{v,\delta}(q) \leqslant - \eta$ as $q \to w$ for any $\measuredangle (w, v) > \delta$. ($\eta = -c$ in construction)
\end{enumerate}
Since $\varphi$ is continuous on $\mathbb{S}^{n-1}$, let $A \coloneqq \sup |\varphi|$. 

\textbf{Step 3:} Let $W_{v, \varepsilon, \delta, \eta} = \frac{2A}{\eta}B_{v, \delta} + \varphi (v) - \varepsilon$ (where $\delta = \delta (\varepsilon), \eta = \eta (v, \delta)$), therefore $W_{v, \varepsilon, \delta, \eta}$ satisfies 
\begin{align*}
    \begin{cases}
    \underset{q \to w}{\limsup} W_{v, \varepsilon, \delta, \eta} \leqslant \varphi(w) \quad \forall \, w \in \mathbb{S}^{n-1}, \\
    \underset{q \to v}{\lim} W_{v, \varepsilon, \delta, \eta} = \varphi(v) - \varepsilon,
    \end{cases}
\end{align*}
and
\[
\frac{\sigma_n}{\sigma_{n-1}} \left[W_{ij}-W\delta_{ij}\right] \geqslant -\frac{1}{n} W \quad \text{in } \mathbb{H}^{n}
\]
in viscosity sense. Since in $\Omega_1$, $B$ is a smooth subsolution by Lemma \ref{f(s)}, and any negative constant is a smooth solution, then after taking their supremum, it is still a subsolution. 
% because $\overline{\Omega}_2 \subset \subset \Omega_1$.

Therefore, $W_{v,\varepsilon,\delta,\eta}$ is a viscosity subsolution of \eqref{k=1 selfshrinker}. Then let
\[
\underline{v}(x) \coloneqq \sup \{W_{v,\varepsilon,\delta,\eta}(x) \mid v, \varepsilon, \delta, \eta\}.
\]
$\underline{v}$ is also a viscosity subsolution, and $\underline{v} = \varphi$ on $\mathbb{S}^{n-1}$, which is what we want. It should also be noted that by comparison principle and $\varphi < 0$, we have $\underline{v} \leqslant \sup \varphi < 0$.
\end{proof}

\section{Self-shrinkers of the inverse Hessian curvature flow}

In the following sections, we solve \eqref{sigmak selfshrinker} by an approximate procedure and convergence argument following the method of \cite{Entireself-expanders}. First, in this section, we solve approximate Dirichlet problem by continuity method.

Let $U_r \coloneqq \{X \in \mathbb{H}^n | x_{n+1} < \frac{1}{\sqrt{1-r^2}}\}$. Then let $B_r=P(U_r)$, where the definition of $P$ is given in Section \ref{Preliminaries}.

\subsection{Approximate Dirichlet problem}\label{Approximate Dirichlet problem}

We first construct a local subsolution between the entire lower barrier and upper barrier directly.
\begin{lemma}
    For any fixed $U_r$, there exists a smooth hyperbolic convex function $\overset{\circ}{v}$ defined in $\overline{U}_r$, satisfying $\underline{v} \leqslant \overset{\circ}{v} \leqslant \overline{v}$, where $\underline{v},\overline{v}$ are viscosity  subsolution and smooth supersolution constructed in Section \ref{The construction of barriers: self-shrinker}.  
\end{lemma}
\begin{proof}
    Since $\underline{v}$ is a viscosity subsolution to equation \eqref{viscosity subsolution equation}, by definition and generalized Newton--Maclaurin's inequality, we can see that $\underline{v}$ is viscosity subharmonic, which also means subharmonic in distribution sense in any bounded set $\Omega$. Let 
    \[
    q = \underline{v} - \overline{v} \leqslant 0 \quad \text{in } \overline{\mathbb{H}}{}^n,
    \]
    $q$ is also subharmonic. If there exists a point $z_0 \in \mathbb{H}^n$ such that $q(z_0) = 0$, then by strong maximum principle, we have $q \equiv 0$. In this case, we take 
    \[
    \overset{\circ}{v} = \underline{v} = \overline{v}.
    \]
    Otherwise, $q<0$ in $\mathbb{H}^n$, then there exists $\delta>0$ such that
    \[
    \overline{v} - \underline{v} \geqslant \delta \quad \text{in } \overline{U}_r.
    \]
    Thus in $B_1$ model, we can translate the graph of $\underline{u}^*(\xi)=\underline{v}(P^{-1}(\xi))\sqrt{1-|\xi|^2}$ upward and smooth it by convolution. It is still a convex function. One may see \cite{Entireconstantsigmak} for more detail about $B_1$ model. This concludes the construction.
\end{proof}

Define 
\[
\tilde{\underline{v}} \coloneqq A\left(x_{n+1} - \frac{1}{\sqrt{1-r^2}}\right) + \overset{\circ}{v} \quad \text{in } U_r.
\]
Then on $\partial U_r$, we have
\[
\underline{v} \leqslant \tilde{\underline{v}} = \overset{\circ}{v} \leqslant \overline{v},
\]
and 
\begin{align*}
    F(\tilde{\underline{v}}) + \frac{\tilde{\underline{v}}}{\binom{n}{k}^{\frac{1}{k}}} &\geqslant A F(x_{n+1} - \frac{1}{\sqrt{1-r^2}}) + F(\overset{\circ}{v}) + \frac{\tilde{\underline{v}}}{\binom{n}{k}^{\frac{1}{k}}} \\
    &\geqslant A\frac{1}{\sqrt{1-r^2}}\frac{1}{\binom{n}{k}^{\frac{1}{k}}} + \frac{A(x_{n+1} - \frac{1}{\sqrt{1-r^2}}) + \overset{\circ}{v}}{\binom{n}{k}^{\frac{1}{k}}} \\
    &= \frac{Ax_{n+1} + \overset{\circ}{v}}{\binom{n}{k}^{\frac{1}{k}}} \geqslant 0,
\end{align*}
if we choose $A \geqslant -\inf \underline{v}$. Hence $\tilde{\underline{v}}$ is a new subsolution.

To use the continuity method, we need to make global estimates up to second order. Fix some $0<r<1$, the approximate Dirichlet problem is 
\begin{equation}
    \begin{cases} 
       \left( \frac{\sigma_n}{\sigma_{n-k}} \left[v^t_{ij}-v^t\delta_{ij}\right] \right)^{\frac{1}{k}} = -\frac{v^t(y)}{\binom{n}{k}^{\frac{1}{k}}} + (1-t)G(\tilde{\underline{v}}) \quad \text{in } U_r, \\
       v^t = \tilde{\underline{v}} \quad \text{on } \partial U_r,
       \label{sigmak approximate equation}
    \end{cases}
\end{equation}
where $\tilde{\underline{v}}$ is a solution when $t=0$. What we want is the solution when $t=1$. If the solution at $t=1$ exists, say $v$, then by comparison principle, we have
\[
\underline{v} \leqslant v \leqslant \overline{v}. 
\]
In the following, we fix and suppress superscript $t$ for convenience.

\subsection{Global \texorpdfstring{$C^0$}{C0} estimates}

By simple calculation, $\overline{v}$ and $\tilde{\underline{v}}$ are upper barrier and lower barrier for the equation \eqref{sigmak approximate equation}, respectively. By Lemma \ref{Comparison principle} and the maximum principle for subharmonic function, we have
\[
\tilde{\underline{v}} \leqslant v \leqslant \overline{v} \quad \text{in } \overline{U}_r
\]
for any solution $v$ to equation \eqref{sigmak approximate equation}, where $U_r$ is the geodesic ball we fixed before. Hence we have global $C^0$ estimates for $v$.

\subsection{Global \texorpdfstring{$C^1$}{C1} estimates}

\begin{lemma}[$C^1$ estimates]\label{C1estimates}
    If $v$ solves \eqref{sigmak approximate equation}, we have
    \[
    |\nabla v(y)|^2 \leqslant C, \quad \forall y \in \overline{U}_r,
    \]
    where $C$ is a uniform constant. 
\end{lemma}

We first obtain an estimate of $\left| D u^* \right|$, where $v(y) = \frac{u^*(\xi)}{w^*(\xi)}$ by projection and $w^*(\xi) = \sqrt{1 - |\xi |^2}$. 
\begin{lemma}[$C^1$ estimates for $u^*$]\label{C1estimates for u*}
    If $v$ solves \eqref{sigmak approximate equation}, $P(U_r) = B_r \subseteq B_1$, we have
    \[
    \left| D u^*(\xi) \right| \leqslant C , \quad \forall \xi \in \overline{B}_r,
    \]
    where $C$ depends only on $|D \tilde{\underline{u}}^*|$.
\end{lemma}
\begin{proof}
    Let $\tilde{\underline{u}}^*(\xi) = \tilde{\underline{v}}(y) \sqrt{1 - |\xi|^2}, u^*(\xi) = v(y) \sqrt{1 - |\xi|^2}$, where $y=P^{-1}(\xi)$. Then
    \[
    \tilde{\underline{u}}^*(\xi) = u^*(\xi) \quad \text{on } \partial B_r,
    \]
    and
    \[
    \tilde{\underline{u}}^*(\xi) \leqslant u^*(\xi) \quad \text{in } B_r.
    \]
    Hence
    \[
    D \tilde{\underline{u}}^*(B_r) \supset D u^*(B_r).
    \]
    It follows that
    \[
    |D u^*(\xi)| \leqslant \sup_{\overline{B}_r} |D \tilde{\underline{u}}^*| \quad \text{in } B_r.
    \]
\end{proof}

\begin{lemma}[Relation between $|\nabla v|$ and $|D u^*|$] \label{Relation}
    We have the following relation between $|\nabla v|$ and $|D u^*|$:
    \begin{align*}
        |\nabla v|^2 - v^2 = |D u^*(\xi)|^2 - (\xi \cdot D u^*(\xi) - u^*(\xi))^2,
    \end{align*}
    where  $u^*(\xi) = v(P^{-1}(\xi)) \sqrt{1 - |\xi|^2}$.
    \end{lemma}
\begin{proof}
    Since $u^*$ is strictly convex, we can let $u$ be its Legendre transform and $x=D u^*$.  
  Thus,  take inner product of position vector $X=(x,u(x))$ with itself, then
    \begin{align*}
        |X|^2 = |x|^2 - u^2(x) = |\nabla v|^2 - v^2.
    \end{align*}
    Hence we get
    \begin{align*}
        |\nabla v(y)|^2 - v^2(y) = |D u^*(\xi)|^2 - (\xi \cdot D u^*(\xi) - u^*(\xi))^2.
    \end{align*}
 \end{proof}

\begin{proof}[Proof of Lemma \ref{C1estimates}]
    By Lemma \ref{Relation}, we have
    \begin{align*}
        |\nabla v|^2 &= v^2 + |D u^*(\xi)|^2 - (\xi \cdot D u^*(\xi) - u^*(\xi))^2 \leqslant C. 
    \end{align*}
    Since $v$ is bounded, $|D u^*|$ is bounded by Lemma \ref{C1estimates for u*}, $|\xi| \leqslant 1$ and $u^* = v w^*$ is bounded.
\end{proof}
Hence we have global $C^1$ estimates for $v$.

\subsection{Boundary \texorpdfstring{$C^2$}{C2} estimates}

Consider linearized operator $\mathscr{L}$ defined by
\begin{align*}
    \mathscr{L} w &\coloneqq F^{ij} \nabla^2_{ij}w - w \sum_{i}F^{ii} = F^{ij} (\Lambda_w)_{ij},
\end{align*}
where $F = \left( \frac{\sigma_n}{\sigma_{n-k}} \right)^{\frac{1}{k}}$ and $F^{ij}$ is defined with respect to $\Lambda_{ij}=v_{ij}-v\delta_{ij}$.

\begin{lemma}\label{auxiliary function for boundary estimates}
    For any positive constant $A$, let
    \[
    h = v - \tilde{\underline{v}} + A \left(\frac{1}{\sqrt{1-r^2}} - x_{n+1}\right).
    \]
    Then we have
    \[
    \mathscr{L} h \leqslant -A \sum_{i} F^{ii} \quad \text{in } U_r,
    \]
    and
    \[
    h \geqslant 0 \quad \text{on } \partial U_r.
    \]
\end{lemma}
\begin{proof}
    We have
    \begin{align*}
        \mathscr{L} h &= F^{ij} ((v - \tilde{\underline{v}})_{ij} - A x_{n+1} \delta_{ij}) - (v - \tilde{\underline{v}}) \sum_{i}F^{ii} - A \left(\frac{1}{\sqrt{1-r^2}} - x_{n+1}\right) \sum_{i}F^{ii} \\
        &\leqslant F[v] - F[\tilde{\underline{v}}] - A\frac{1}{\sqrt{1-r^2}} \sum_{i}F^{ii} \\
        &= -\frac{1}{\binom{n}{k}^{\frac{1}{k}}}v + (1-t)g - g + \frac{1}{\binom{n}{k}^{\frac{1}{k}}}\tilde{\underline{v}} - A\frac{1}{\sqrt{1-r^2}}\sum_{i}F^{ii} \\
        &\leqslant -tg - A\frac{1}{\sqrt{1-r^2}} \sum_{i}F^{ii} \\
        &\leqslant -A \sum_{i} F^{ii},
    \end{align*}
    where the first inequality is from the concavity of $F$.
\end{proof}
Then the same proof can be carried out following the argument in \cite{GuanJiao} to derive $C^2$ boundary estimates. For the double normal estimates, one may proceed analogously as in \cite{GuanJiao}, with the caveat that here the tensor $A[u] = -ug$ does not satisfy (1.10) of \cite{GuanJiao}. Nevertheless, we can still prove Proposition 2.1 in \cite{GuanJiao} directly in $M_\delta$, which suffices for constructing the barriers needed in the boundary estimates.
\begin{lemma}[Proposition 2.1 in \cite{GuanJiao}]
    There exist uniform positive constants $R, \theta, \delta$ such that
    \[
    \mathscr{L}\left( \tilde{\underline{v}} - v \right) \geqslant \theta \sum F^{ii} + \theta, \quad \text{whenever } |\lambda(\Lambda)| \geqslant R
    \]
    in $U_{r,\delta} = \{X \in \mathbb{H}^n | \frac{1}{\sqrt{1-r^2}} - \delta \leqslant x_{n+1} \leqslant \frac{1}{\sqrt{1-r^2}}\}$.
\end{lemma}
\begin{proof}
    As (2.7) in Proposition 2.1 of \cite{GuanJiao}, we have
    \begin{align}
        F^{ii}(\underline{\Lambda}_{ii} - \Lambda_{ii}) \geqslant F(\underline{\Lambda}) - F(\Lambda) + \theta\sum F^{ii} + \theta,
    \end{align}
    where $\underline{\Lambda} = \Lambda_{\tilde{\underline{v}}}$. Hence
    \begin{align*}
        \mathscr{L}\left( \tilde{\underline{v}} - v \right) &= 
        F^{ii}(\underline{\Lambda}_{ii} - \Lambda_{ii}) 
        \geqslant - \frac{\tilde{\underline{v}} - v}{\binom{n}{k}^{\frac{1}{k}}} + \theta\sum F^{ii} + \theta \\
        &\geqslant -C|\nabla (\tilde{\underline{v}} - v)|\delta + \theta\sum F^{ii} + \theta \geqslant \frac{\theta}{2} + \theta\sum F^{ii}.
    \end{align*}
    provided $\delta>0$ is sufficiently small, since we already have gradient estimates.
\end{proof}

\subsection{Global \texorpdfstring{$C^2$}{C2} estimates}

For global $C^2$ estimates, we will use the following lemma. The proof is the same as Lemma 18 in \cite{Entireself-expanders}.

\begin{lemma}[Lemma 18 in \cite{Entireself-expanders}]\label{C2 global estimates for selfshrinker}
    Let $v$ be a solution to approximate Dirichlet problem \eqref{sigmak approximate equation}. Denote the eigenvalues of $\left[ g^{-1} \circ \left( \nabla^2 v - vg \right) \right]$ by $(\lambda_1, \cdots, \lambda_n)$ with $\lambda_1 \geqslant \lambda_2 \geqslant \cdots \geqslant \lambda_n > 0$. Then 
    \[
    \lambda_1 \leqslant \max\left\{ C, \lambda_{\partial U_r} \right\},
    \]
    and $C$ is a positive constant depending only on $U_r$ and $\| v \|_{C^0(\overline{U}_r)}$.
\end{lemma}

Then since $\nabla^2 v - vg$ is positive definite, we get lower bounds for $\nabla^2 v$. Hence we have global $C^2$ estimates for $v$.

%\section{Self-shrinkers of inverse Gauss curvature flow}

\section{Convergence of local solutions to an entire solution: Gauss curvature case}

In this section, we will prove the existence of self-shrinkers of inverse Gauss curvature flow with given boundary condition, i.e.
\begin{align}
    \begin{cases}
        \left( {\sigma_n} \left[v_{ij}-v\delta_{ij}\right] \right)^{\frac{1}{n}} = - v \quad \text{in } \mathbb{H}^{n}, \\
        v \rightarrow \varphi \quad \text{on } \mathbb{S}^{n-1} = \mathbb{H}^{n}(\infty). \label{Gauss selfshrinker}
    \end{cases}
\end{align}
It is the special case when $k=n$ in the equation \eqref{sigmak selfshrinker}. The reason why we consider this equation at first is because the upper barrier $\overline{v}$ we constructed before is not hyperbolic convex, thus we cannot take Legendre transform of $\overline{v}$. However, the Gauss curvature self-shrinker is a supersolution to the equation \eqref{sigmak selfshrinker} for any $k$ still by generalized Newton--Maclaurin's inequality. Hence we can use it as a new upper barrier to get local estimates for dual equation later.

By previous estimates or directly by Theorem 1.3 in \cite{GuanBo}, we know that the Dirichlet problem
\begin{equation}
    \begin{cases} 
        \left( \sigma_n \left[v_{ij}-v\delta_{ij}\right] \right)^{\frac{1}{n}} = -v \quad \text{in } U_r, \\
        v = \tilde{\underline{v}} \quad \text{on } \partial U_r. \label{Ur Gauss selfshrinker}
    \end{cases}
\end{equation}
has a unique solution $v_r$. 
We will show that a subsequence of $\left\{ v_r \right\}$ converges to an entire solution to equation \eqref{Gauss selfshrinker} locally uniformly as $r \to 1$. 
To prove this, we first obtain local estimates for $v_{r^{\prime}}$ with $r^{\prime}>r$ in a fixed geodesic ball $U_r$ in hyperbolic space, i.e. estimates for solutions $v_{r^{\prime}}$ restricted in $U_r$, and the constant depends on $r$.

In the following, we always denote $u^*(\xi)=v(P^{-1}(\xi))\sqrt{1-|\xi|^2}$.  $\underline{u}^*, u^*_{r^{\prime}}, \overline{u}^*$ have the similar definition.  

\subsection{Local \texorpdfstring{$C^0$}{C0} estimates}

Since $\overline{v}$ and $\underline{v}$ and upper barrier and lower barrier of equation \eqref{Ur Gauss selfshrinker}, we have
\[
\underline{v} \leqslant v_{r^{\prime}} \leqslant \overline{v} \quad \text{in } \overline{U}_r.
\]
Hence
\[
\underline{u}^* \leqslant u^*_{r^{\prime}} \leqslant \overline{u}^* \quad \text{in } \overline{B}_r
\]
for any $r^{\prime}>r>0$.

\subsection{Local \texorpdfstring{$C^1$}{C1} estimates}\label{Gauss Local C1}

Since $u^*$ is convex, then by local Lipschitz property of convex functions, $|D u^*|$ is locally uniformly bounded when given $C^0$ bound. Recall 
\[
|u^*(\xi)| = |v(y(\xi))| \sqrt{1 - |\xi |^2} \leqslant C.
\]
Then we have
\[
|D u^*_{r^{\prime}}| \leqslant C(\delta) \quad \text{in } \overline{B}_r,
\]
if $r^{\prime} - r \geqslant \delta$ for some $\delta > 0$. Hence by the same argument as Lemma \ref{Relation}, we obtain
\[
|\nabla v_{r^{\prime}}| \leqslant C(\delta) \quad \text{in } \overline{U}_r.
\]

\subsection{Local \texorpdfstring{$C^2$}{C2} estimates}\label{Gauss Local C2}

Here we want a Pogorelov-type interior $C^2$ estimate. Let $v_r$ denote the solution to the approximate Dirichlet problem in $\overline{U}_r = \{X \in \mathbb{H}^n \mid x_{n+1} \leqslant \frac{1}{\sqrt{1-r^2}}\}$. Then when $r$ tends to $1$, the definition domain of $v_r$ is approaching the whole space $\mathbb{H}^n$. For convenience, we will denote the restriction of the function $v_{r^{\prime}}$ to $U_r$ by $v$ in the following.

Since $\overline{v}$ and $\underline{v}$ are both asymptotic to $\varphi$ at infinity, we have $\overline{u}^*(\xi), \underline{u}^*(\xi) \to 0$ uniformly as $|\xi| \to 1$, i.e.
\[
\underline{u}^*(\xi) = \overline{u}^*(\xi) = 0 \quad \text{on } \partial B_1.
\]
Thus we can do cutoff for $u^*$ by $-\varepsilon$. Define 
\[
B_{\varepsilon} \coloneqq \{\xi \in B_1 \mid -\varepsilon - u^*(\xi) > 0\}, 
\]
and
\[
A_{\varepsilon} \coloneqq P^{-1} (B_{\varepsilon}) = \{y \in \mathbb{H}^n \mid -\varepsilon - \frac{v}{x_{n+1}}(y) > 0\}.
\]
% Let's recall the following geometric formulae:
% \begin{align}
% \Lambda_{ij,k} &= \Lambda_{ik,j}  \\
% \Lambda_{lk,ji} - \Lambda_{lk,ij} &= v_{lkji} - v_{lkij} \notag \\
% &= -v_{lj}\delta_{ik} + v_{li}\delta_{jk} - v_{jk}\delta_{il} + v_{ik}\delta_{jl}.
% \end{align}
% \Lambda_{11,ii} = \Lambda_{ii,11} + \Lambda_{ii} - \Lambda_{11}
\begin{lemma}\label{Gauss selfshrinker Pogorelov}
    Given $\varepsilon>0$ close to $0$, for $r$ sufficiently close to $1$ such that 
    \[
    \left. \frac{v}{x_{n+1}}\right|_{\partial U_r} > -\varepsilon.
    \]
    Then we have
    \[
    \max_{A_{\varepsilon}} \left( -\varepsilon -  \frac{v}{x_{n+1}} \right)^{\alpha} \lambda_1 \leqslant C.
    \]
    where $\lambda_1$ denotes the maximum eigenvalue of $v_{ij}-v\delta_{ij}$ and $\alpha$ is a large positive constant.
\end{lemma}

The proof is an elliptic analog of Lemma 28 in \cite{Entireconvexcurvatureflow}, we include the details in Appendix \ref{Proof of Gauss selfshrinker Pogorelov}.

Thus we can show the existence of the solution to the Gauss curvature self-shrinker equation \eqref{Gauss selfshrinker} by a standard argument (see section \ref{Proof of Theorem 1.1} for more detail). By Appendix \ref{Recover hypersurface from its support function}, we can see that the above solution is indeed the support function of an entire spacelike strictly convex hypersurface.

\section{Convergence of local solutions to an entire solution: general case}

We still denote $v_r$ the unique solution to the Dirichlet problem 
\begin{equation}
    \begin{cases} 
       \left( \frac{\sigma_n}{\sigma_{n-k}} \left[v_{ij}-v\delta_{ij}\right] \right)^{\frac{1}{k}} = -\frac{v}{\binom{n}{k}^{\frac{1}{k}}} \quad \text{in } U_r, \\
       v = \tilde{\underline{v}} \quad \text{on } \partial U_r.
    \end{cases}
\end{equation}
In the following, we will deal with the dual curvature equation of \eqref{sigmak selfshrinker}
\begin{align}
    \left( \sigma_k \left[\frac{1}{w} \gamma^{ik} u_{kl} \gamma^{lj}\right] \right)^{\frac{1}{k}} &= -\binom{n}{k}^{\frac{1}{k}} \left( \frac{1}{\sqrt{1 - |Du|^2}} \left( \sum_i x_i \frac{\partial u}{\partial x_i} - u \right) \right)^{-1} \notag \\
    &= -\binom{n}{k}^{\frac{1}{k}} \left( \langle X, \nu \rangle \right)^{-1}. \label{sigmak curvature equation}
\end{align}
We can express this equation in the following form
\begin{align}
    \sigma_k(\kappa[\mathcal{M}_u]) = \psi(X, \nu) \label{sigmak curvature equation for psi},\ \  
    \psi(X, \nu) = \frac{\binom{n}{k}}{\left( -\langle X, \nu \rangle \right)^k}. 
\end{align}
Let $u_r^*(\xi) = v_r(y(\xi)) \sqrt{1 - |\xi |^2}$ and $u_r$ be the Legendre transform of $u_r^*$. Then $u_r$ is a solution to equation \eqref{sigmak curvature equation for psi} in $D u_r^*(B_r)$, where $B_r = P(U_r)$ is the ball in $B_1$ model.

In this section, we will show that a subsequence of $\left\{ u_r \right\}$ converges to an entire solution to equation \eqref{sigmak curvature equation for psi} locally uniformly as $r \to 1$. 
To prove this, we first obtain local estimates for $u_{r}$ in any given compact set $K \subset \subset \mathbb{R}^n$.

In the following, we replace the upper barrier $\overline{v}$ constructed in Section \ref{The construction of barriers: self-shrinker} with the solution to Gauss curvature self-shrinker equation \eqref{Gauss selfshrinker}, still denoted by $\overline{v}$. The benefit of this change is that  $\overline{v}$ is hyperbolic convex, so we can take Legendre transform.

\subsection{Local \texorpdfstring{$C^0$}{C0} estimates}\label{Legendre transform}

By maximum principle, we have 
\[
\underline{v} \leqslant v_r \leqslant \overline{v} \quad \text{in } U_r.
\]
Then define $\underline{u}^*(\xi) = \underline{v}(y(\xi)) \sqrt{1 - |\xi |^2}$ and $\overline{u}^*(\xi) = \overline{v}(y(\xi)) \sqrt{1 - |\xi |^2}$, we have
\begin{align*}
    \underline{u}^* \leqslant u_r^* \leqslant \overline{u}^* \quad \text{in } B_r.
\end{align*}
Let $\underline{u}$, $u_r$ and $\overline{u}$ be the Legendre transform of $\overline{u}^*$ , $u_r^*$ and $\underline{u}^*$ respectively. By the same argument as in \cite[section 6.1]{Entireconstantsigmak}, we can show that for any $K \subset \subset \mathbb{R}^n$, we have
\begin{align*}
    \underline{u} \leqslant u_r \leqslant \overline{u} \quad \text{in } K
\end{align*}
for $r>r_K$ is sufficiently large.

\subsection{Local \texorpdfstring{$C^1$}{C1} estimates}

\begin{lemma}[Lemma 5.1 in \cite{BayardSchnürer}]\label{Local C1}
Let $\Omega \subseteq \mathbb{R}^n$ be a bounded open set. Let $u, \overline{u}, \psi : \Omega \to \mathbb{R}$ be strictly spacelike. Assume that near $\partial\Omega$, we have $\psi > \overline{u}$. Everywhere in $\Omega$ we assume that $u \leqslant \overline{u}$. Consider the set, where $u > \psi$. For every $x$ in that set, we get the following gradient estimate for $u$:
\[
\frac{1}{\sqrt{1 - |Du(x)|^2}} \leqslant \frac{1}{u(x) - \psi(x)} \cdot \sup_{\{u > \psi\}} \frac{\overline{u} - \psi}{\sqrt{1 - |D\psi|^2}}.
\]
\end{lemma}

Hence we only need to construct a suitable spacelike function $\psi$. With the same proof of Lemma 14 in \cite{Entireconstantsigmak}, we can show that $\underline{u}, \overline{u} \rightarrow \left| x \right|$ uniformly as $\left| x \right| \to +\infty$. Let $\delta>0$ be a constant and $\psi = \underline{u} + \delta$. Therefore there exists a bounded open set $\Omega$ such that $\psi > \overline{u}$ near the boundary of $\Omega$. Then smoothing $\psi$ by convolution concludes the construction.

\subsection{Local \texorpdfstring{$C^2$}{C2} estimates}

\begin{lemma}[Lemma 17 in \cite{Theprescribedcurvatureproblem}, Lemma 21 in \cite{Entireself-expanders}]
Let $u_r$ be the same as before. For any given $s > 2C_0 + 1$, where $C_0 > \min \overline{u}$ is an arbitrary constant, let $r_s > 0$ be a positive number such that, when $r > r_s$, we have $u_r |_{\partial \Omega_r} > s$, where $\Omega_r = D u_r^*(B_r)$. Let $\kappa_{\max}(x)$ be the largest principal curvature of $\mathcal{M}_{u_r}$ at $x$, where $\mathcal{M}_{u_r} = \{ (x, u_r(x)) \mid x \in \Omega_r \}$. Then, for $r > r_s$, we have
\begin{equation}
\max_{\mathcal{M}_{u_r}} \left( s-u_r \right) \kappa_{\max} \leqslant C.
\end{equation}
Here, $C$ depends on the local $C^1$ estimates of $u_r$ and $s$.
\end{lemma}

\begin{proof}
    The proof is the same as that of Lemma 17 in \cite{Theprescribedcurvatureproblem}, with $\psi$ here is 
    \[
    \psi(X, \nu) = \frac{\binom{n}{k}}{\left( -\langle X, \nu \rangle \right)^k}. 
    \]
\end{proof}

\subsection{Proof of Theorem \ref{self-shrinker existence}}\label{Proof of Theorem 1.1}

% Combine the estimates above, we can conclude to give the proof.

\begin{proof}[Proof of Theorem \ref{self-shrinker existence}]
    Choose a sequence of positive numbers $r_i \to 1$ and consider the corresponding sequence of solutions $u_{r_i}$ to equation \eqref{sigmak curvature equation for psi} in $\Omega_{r_i}$ constructed by Legendre transform of the solution $u_{r_i}^*$ to the corresponding approximate Dirichlet problem. Then by local $C^0$, $C^1$ and $C^2$ estimates, with the Cantor's diagonal argument, there exists a subsequence $\left\{ u_{r_i} \right\}$ that converges to some function $u_{\infty}$ in $C^{2,\alpha}_{\mathrm{loc}}$ topology. Then by standard regularity theory for elliptic PDE, we have that $u_{\infty}$ is a smooth solution to equation \eqref{sigmak curvature equation for psi}. 
    Moreover, by local $C^0$ estimates, we have 
    \[
    \underline{u} \leqslant u_{\infty} \leqslant \overline{u} \quad \text{in } \mathbb{R}^n.
    \]
    By the same argument in \cite[Theorem 29]{Entireconstantsigmak}, we obtain that the graph of $u_{\infty}$ will split as a product $\mathcal{M}^l \times \mathbb{R}^{n-l}$. But since $\underline{u}$ and $\overline{u}$ are both strictly convex, $n-l$ must equal to $0$, i.e. $u_{\infty}$ is strictly convex.

    By definition of Legendre transform and the same argument as in \cite[section 6.1]{Entireconstantsigmak}, we have
    \[
    \underline{u}^* \leqslant u_{\infty}^* \leqslant \overline{u}^* \quad \text{in } B_1.
    \]
    Multiplying $\frac{1}{\sqrt{1 - |\xi|^2}}$ on both sides, we have
    \[
    \underline{v} \leqslant v_{\infty} \leqslant \overline{v} \quad \text{in } \mathbb{H}^n.
    \]
    From the definition of $\underline{v}$ and $\overline{v}$, we have $\underline{v} = \overline{v} = \varphi$ on $\mathbb{H}^n(\infty)$. Then 
    \[
    v_{\infty} = \varphi \quad \text{on } \mathbb{H}^n(\infty).
    \]
    Hence $v_{\infty}$ is the desired solution to equation \eqref{sigmak selfshrinker} and $u_{\infty}$ is the corresponding entire hypersurface satisfying \eqref{sigmak curvature equation for psi}.
\end{proof}

\section{The inverse Hessian curvature flow}

Now we consider evolution equation for support function of the inverse $\sigma_k$ curvature flow
\begin{align}
    \begin{cases}
    v_t = F(\Lambda_{v}) & \text{in } P^{-1}(B_1) \times (0, \infty) \coloneqq U_1 \times (0, \infty), \\[4pt]
    v(\cdot, t) = V(t) \varphi & \text{on } \partial U_1 \times [0, \infty), \\[4pt]
    v(\cdot, 0) = u_0^* x_{n+1} & \text{on } U_1 \times \{0\}. \label{support function flow}
    \end{cases}
  \end{align}
where $F=\left( \frac{\sigma_n}{\sigma_{n-k}} \right)^{\frac{1}{k}}$ and $u_0^*$ is the Legendre transform of $u_0$, the height function of initial hypersurface. Then $v_0=u_0^* x_{n+1}$ is the support function of initial hypersurface. In fact $U_1 = \mathbb{H}^n$ and $\partial U_1 = \mathbb{H}^n(\infty)$.

Let $V(t) = e^{-\gamma t},  \gamma = \frac{1}{\binom{n}{k}^{\frac{1}{k}}}$.
Consider the normalized flow of \eqref{support function flow}, $\tilde{v}=V(t)^{-1}v=e^{\gamma t}v$, then we have 
\begin{align}
    \begin{cases}
    \tilde{v}_t = F(\Lambda_{\tilde{v}}) + \frac{\tilde{v}}{\binom{n}{k}^{\frac{1}{k}}} & \text{in } U_1 \times (0, \infty), \\[4pt]
    \tilde{v}(\cdot, t) = \varphi & \text{on } \partial U_1 \times [0, \infty), \\[4pt]
    \tilde{v}(\cdot, 0) = v_0 & \text{on } U_1 \times \{0\}. \label{normalized support function flow}
    \end{cases}
\end{align}
Since the normalized flow \eqref{normalized support function flow} is equivalent to flow \eqref{support function flow} when considering long-time existence, we will only deal with the normalized flow in the following sections.

\subsection{Comparison principle}

The proof of asymptotic comparison principle is same as Lemma \ref{Comparison principle}, since $\overline{\mathbb{H}}{}^n \times [0,T]$ is endowed with the product topology.

\section{The construction of barriers: inverse Hessian curvature flow}

\subsection{Construction of upper barrier}

For upper barrier, consider the solution to the following heat equation
\begin{align}
    \begin{cases}
    v_t = \frac{\Delta v}{n \cdot \binom{n}{k}^{\frac{1}{k}}} & \text{in } U_1 \times (0, \infty), \\[4pt]
    v(\cdot, t) = \varphi & \text{on } \partial U_1 \times [0, \infty), \\[4pt]
    v(\cdot, 0) = v_0 & \text{on } U_1 \times \{0\}. \label{heat equation}
    \end{cases}
\end{align}
Its solution is denoted by $\overline{v}$. For completeness, we show the existence of $\overline{v}$ as follows. First, under the time rescaling $s = \alpha t$, we can assume that it is the standard heat equation $v_t = \Delta v$. We claim that 
\[
\overline{v} = h + w,
\]
where $h$ is the harmonic extension of $\varphi$ and $w$ can be formulated by convolution with the heat kernel, i.e.
\[
w(x, t) = e^{t\Delta} \left[ v_0 - h \right](x) = \int_{\mathbb{H}^n} p_n(d(x,y), t)(v_0 - h)(y) d\mu(y).
\]
We only need to check that $w$ vanishes at infinity. Let $w_0 = v_0 - h$, since $w_0 = 0$ on $\mathbb{H}^n(\infty)$, there exists a $B_R$ such that $|w_0| < \varepsilon$ outside $B_R$, then we have
\begin{align*}
    &|w(x, t)| \\
    \leqslant& \left| \int_{\mathbb{H}^n \setminus B_R} p_n(d(x,y), t)w_0(y) d\mu(y) \right| + \left| \int_{B_R} p_n(d(x,y), t)w_0(y) d\mu(y) \right| \\
    \leqslant& \varepsilon + \varepsilon = 2\varepsilon,
\end{align*}
as $x \to \mathbb{H}^n(\infty)$ if $t>0$. The estimate for the second term follows from the asymptotic behavior of $p_n$ when $d(x,y)$ tends to infinity. The heat kernel on hyperbolic space has an explicit expression in \cite[Theorem 1.1]{TheHeatKernel} and the estimates for heat kernel can be found in \cite[Theorem 3.1]{HeatKernelBoundsonHyperbolicSpaceandKleinianGroups} that 
\[
p_{n+1}(\rho, t) \sim t^{-\frac{1}{2}(n+1)} e^{-\frac{1}{4}n^2 t - \rho^2/4t - \frac{1}{2}n\rho} \cdot (1 + \rho + t)^{\frac{1}{2}n-1}(1 + \rho),
\]
uniformly for $0 \leqslant \rho < \infty$ and $0 < t < \infty$. The heat kernel estimates also yields that as $t \to +\infty$,
\[
\overline{v}(\cdot, t) \rightarrow h \quad \text{uniformly on } \mathbb{H}^n.
\]

Then by generalized Newton--Maclaurin's inequality, we have
\begin{align*}
    \left( \frac{\sigma_n}{\sigma_{n-k}} \right)^{\frac{1}{k}} \left( v_{ij}-v\delta_{ij} \right) &\leqslant \frac{1}{\binom{n}{k}^{\frac{1}{k}}} \frac{\sigma_1}{n} \left( v_{ij}-v\delta_{ij} \right)
    = \frac{1}{\binom{n}{k}^{\frac{1}{k}}} \left( \frac{\Delta v}{n} - v \right) \Rightarrow \\
    F(\Lambda_{v}) + \frac{v}{\binom{n}{k}^{\frac{1}{k}}} &\leqslant \frac{\Delta v}{n \cdot \binom{n}{k}^{\frac{1}{k}}}.
\end{align*}
Namely, solution to \eqref{normalized support function flow} is a subsolution to \eqref{heat equation}. Thus by asymptotic maximum principle for heat equation, we have $v \leqslant \overline{v}$.

\subsection{Construction of lower barrier}\label{subsolution to the approximate problem}

From the subsolution condition at infinity (Definition \ref{Cond}), suppose that outside some $U_{r_0} \supset K$, $v_0$ satisfies the subsolution condition. Let $\rho$ be the distance function from the origin. If $g = g(\rho)$ is a smooth function on $[0, \infty)$, then we have 
\begin{align*}
    \nabla^2 g(\rho) = g^{\prime\prime}(\rho)\,\mathrm{d}\rho\otimes \mathrm{d}\rho + g^{\prime}(\rho)\coth \rho \bigl(g-\mathrm{d}\rho\otimes \mathrm{d}\rho\bigr).
\end{align*}
Similarly to Section \ref{The construction of barriers: self-shrinker}, recall that we replace the distance function from a convex set with the distance function from origin here, then we impose the following conditions on $g$:
\[
\begin{cases}
    g: [0, \infty) \to \mathbb{R^{-}}, \\
    \lim_{\rho \to +\infty} g(\rho) \to 0, \\
    g^{\prime}\coth \rho - g > 0, \\
    g^{\prime \prime} - g > 0, \\
    \left[ \frac{\left( 1-\frac{g^{\prime} \coth \rho}{g} \right) \left( 1-\frac{g^{\prime \prime}}{g} \right)}{\left( n-1 \right)\left( 1-\frac{g^{\prime \prime}}{g} \right) -\frac{g^{\prime}}{g} \coth \rho +1} - \frac{1}{n} \right] \geqslant 0.
\end{cases} 
\]
\begin{lemma}\label{smooth function g}
    For $0 < a \leqslant \frac{n-1}{n}$, the function 
    \begin{align}\label{g=}
        g(\rho) = - \left( \cosh \rho \right)^{-a}
    \end{align}
    satisfies the above conditions, i.e. $g$ is admissible and is a subsolution for the self-shrinker equation. Moreover, $g$ is hyperbolic convex for any $a>-1$ and $g$ is a strict subsolution for $0 < a < \frac{n-1}{n}$.
\end{lemma}
\begin{proof}
    By simple calculation, we have
    \begin{align*}
        g^{\prime}(\rho) &= a \left( \cosh \rho \right)^{-a-1} \sinh \rho = a \left( \cosh \rho \right)^{-a} \tanh \rho, \\
        g^{\prime \prime}(\rho) &= -a^2 \left( \cosh \rho \right)^{-a} \tanh^2 \rho + a \left( \cosh \rho \right)^{-a} \operatorname{sech}^2 \rho \\
        &= a \left( \cosh \rho \right)^{-a} \left( \operatorname{sech}^2 \rho - a \tanh^2 \rho \right).
    \end{align*}
    Hence we have
    \begin{align*}
        g^{\prime}\coth \rho - g &= a \left( \cosh \rho \right)^{-a} \tanh \rho \coth \rho + \left( \cosh \rho \right)^{-a} = (a+1) \left( \cosh \rho \right)^{-a} > 0, \\
        g^{\prime \prime} - g &= \left( \cosh \rho \right)^{-a} \left( 1 + a \operatorname{sech}^2 \rho - a^2 \tanh^2 \rho \right) \\
        % &= \left( \cosh \rho \right)^{-a} \left( 1 + a \operatorname{sech}^2 \rho - a^2 \left( 1 - \operatorname{sech}^2 \rho \right) \right) \\
        &= \left( \cosh \rho \right)^{-a} \left( 1 - a^2 + \left( a + a^2 \right) \operatorname{sech}^2 \rho \right) > 0,
    \end{align*}
    if $a>-1$.

    Let
    \begin{align*}
        A &= 1-\frac{g^{\prime} \coth \rho}{g} = a+1, \\
        B &= 1-\frac{g^{\prime \prime}}{g} = 1 - a^2 + \left( a + a^2 \right) \operatorname{sech}^2 \rho \geqslant 1-a^2 \eqqcolon \hat{B}.
    \end{align*}
    Therefore for the last condition, the expression can be simplified to 
    \begin{align*}
        \frac{AB}{\left( n-1 \right)B + A} \geqslant \frac{1}{n}.
    \end{align*}
    Since the LHS is increasing in $B$, we only need to check
    \begin{align*}
        \frac{\left( a+1 \right)\left( 1-a^2 \right)}{\left( n-1 \right)\left( 1-a^2 \right) + \left( a+1 \right)} \geqslant \frac{1}{n}.
    \end{align*}
    This is equivalent to
    \begin{align*}
        n \left( 1-a^2 \right) \geqslant n - \left( n-1 \right) a,
    \end{align*}
    i.e.
    \begin{align*}
                a \leqslant \frac{n-1}{n}.
    \end{align*}
\end{proof}
\begin{remark}
    Indeed, in hyperboloid model, $\cosh \rho = x_{n+1}$.
\end{remark}

Choose some $a < \frac{n-1}{n}$ and fix, then the term in last condition is strictly positive for any $\rho>0$. Hence $F(\Lambda_{g}) + \frac{g}{\binom{n}{k}^{\frac{1}{k}}}$ has a positive lower bound $c_0>0$ in $U_{r_0}$. 

Then define the subsolution as
\[
\underline{v} \coloneqq A\chi(t)g + v_0 \left( A \gg 1 \right),
\]
where $\chi \in C^{\infty}([0, +\infty))$ a nonnegative function such that $\chi(t) = t$ for $0 \leqslant t \leqslant \frac{1}{2}$, $\chi(t) = 1$ for $t \geqslant 1$ and $0\leqslant \chi^{\prime}(t) \leqslant 2$. $A$ is a constant to be determined later. We can see that $\underline{v}$ does not change with time after $t \geqslant 1$. It follows from the expression of $\underline{v}$ that $\underline{v}$ is hyperbolic convex, $\underline{v}(\cdot, 0) = v_0$ and $\underline{v}(\cdot, t) = v_0 = \varphi$ on $\mathbb{H}^n(\infty)$. 
Then the following holds:

In $\mathbb{H}^n \setminus U_{r_0}$, we have
\begin{align*}
    &F(\Lambda_{\underline{v}}) + \frac{\underline{v}}{\binom{n}{k}^{\frac{1}{k}}} 
    \geqslant A\chi F(\Lambda_{g}) + F(\Lambda_{v_0}) + \frac{A\chi g + v_0}{\binom{n}{k}^{\frac{1}{k}}}
    > 0 \geqslant A\chi^{\prime}g = \underline{v}_t,
\end{align*}
by the construction of $g$ and subsolution condition at infinity of $v_0$.

In $U_{r_0}$, 

(1): for $0 \leqslant t \leqslant \frac{1}{2}$, we have
\begin{align*}
    &F(\Lambda_{\underline{v}}) + \frac{\underline{v}}{\binom{n}{k}^{\frac{1}{k}}} - A\chi^{\prime}g 
    \geqslant At F(\Lambda_{g}) + F(\Lambda_{v_0}) + \frac{Atg + v_0}{\binom{n}{k}^{\frac{1}{k}}} - Ag \\
    >& -Ag + \frac{v_0}{\binom{n}{k}^{\frac{1}{k}}}
    \geqslant -Ag(l_0) + \frac{v_0}{\binom{n}{k}^{\frac{1}{k}}} 
    \geqslant 0,
\end{align*}
if $A$ is large, where $l_0$ is the geodesic radius of geodesic ball $U_{r_0}$.

(2): for $t \geqslant \frac{1}{2}$, we have
\begin{align*}
    &F(\Lambda_{\underline{v}}) + \frac{\underline{v}}{\binom{n}{k}^{\frac{1}{k}}} 
    \geqslant A\chi F(\Lambda_{g}) + F(\Lambda_{v_0}) + \frac{A\chi g + v_0}{\binom{n}{k}^{\frac{1}{k}}} \\
    >& \frac{1}{2} c_0 A + \frac{v_0}{\binom{n}{k}^{\frac{1}{k}}} 
    \geqslant 0 \geqslant A\chi^{\prime}g = \underline{v}_t,
\end{align*}
if $A$ is large.

In a word, given $A$ is sufficiently large (depending on $r_0$ and $c_0$), $\underline{v}$ satisfies
\begin{align*}
    \begin{cases}
    \underline{v}_t \leqslant F(\Lambda_{\underline{v}}) + \frac{\underline{v}}{\binom{n}{k}^{\frac{1}{k}}} & \text{in } U_1 \times (0, T], \\[4pt]
    \underline{v}(\cdot, t) = \varphi & \text{on } \partial U_1 \times [0, T], \\[4pt]
    \underline{v}(\cdot, 0) = v_0 & \text{on } U_1 \times \{0\}.
    \end{cases}
\end{align*}
Thus by asymptotic maximum principle, we have
$
\underline{v} \leqslant v.
$

\section{Longtime existence of the inverse \texorpdfstring{$\sigma_k$}{sigmak} curvature flow}\label{Longtime existence}

\subsection{Approximate Dirichlet problem}

We consider the approximate problem of \eqref{normalized support function flow} in fixed domain $U_r$ when $r$ is large, 
\begin{align}\label{approximate normalized support function flow}
    \begin{cases}
    (v_r)_t = F(\Lambda_{v_r}) + \frac{v_r}{\binom{n}{k}^{\frac{1}{k}}} & \text{in } U_r \times (0, T], \\[4pt]
    v_r(\cdot, t) = \underline{v} = A\chi(t)g + v_0 & \text{on } \partial U_r \times [0, T], \\[4pt]
    v_r(\cdot, 0) = \underline{v} & \text{on } U_r \times \{0\}.
    \end{cases}
\end{align}

\subsection{\texorpdfstring{$C^0$}{C0} estimates}

Since $\underline{v}$ and $\overline{v}$ are lower barrier and upper barrier respectively, by asymptotic maximum principle we have
\begin{align*}
    \underline{v} \leqslant v_r \leqslant \overline{v}.
\end{align*}

\subsection{\texorpdfstring{$C^1$}{C1} estimates}

Recall subsolution $\underline{v}$ satisfies
\begin{align}
    \begin{cases}
    \underline{v}_t \leqslant F(\Lambda_{\underline{v}}) + \frac{\underline{v}}{\binom{n}{k}^{\frac{1}{k}}} & \text{in } U_r \times (0, T], \\[4pt]
    \underline{v}(\cdot, t) = \underline{v} & \text{on } \partial U_r \times [0, T], \\[4pt]
    \underline{v}(\cdot, 0) = \underline{v} & \text{on } U_r \times \{0\}.
    \end{cases}
\end{align}
for $r$ large such that $U_r \supseteq K$. Then by maximum principle, we have
\begin{align*}
    \underline{v} \leqslant v_r \quad \forall \, (y,t) \in \overline{U}_r \times [0, T].
\end{align*}
Multiplying $\sqrt{1-|\xi|^2}$ on both sides, we have
\begin{align*}
    \underline{u}^* \leqslant u^*_r \quad \forall \, (\xi,t) \in \overline{B}_r \times [0, T], \text{ and  }
     \underline{u}^* = u^*_r \quad \forall \, (\xi,t) \in \partial B_r \times [0, T].
\end{align*}
Therefore for any $t \in [0, T]$, by convexity we have
\begin{align*}
    |Du^*_r| \leqslant \max_{\xi \in \partial B_r} |D \underline{u}^*|.
\end{align*}
Then still by Lemma \ref{Relation}, we have 
\begin{align*}
    |\nabla v_r| \leqslant C,
\end{align*}
where $C$ is independent of time since $\underline{v}$ does not change with time after $t \geqslant 1$.

\subsection{Upper and lower bounds for \texorpdfstring{$F$}{F}}

For convenience, we omit the subscript when considering approximate problem \eqref{approximate normalized support function flow}.

\begin{lemma}\label{Upper and lower bounds for F}
    Let $v$ be the solution to approximate equation \eqref{approximate normalized support function flow}. Then $F$ is bounded between two positive constants in $\overline{U}_r \times [0, T]$ for $r>\tilde{r}$, where $\tilde{r}$ is a sufficiently large constant.   
\end{lemma}
\begin{proof}
    Denote $w = -v >0$, then we have
    \begin{align*}
        -\left( w_t - F^{ij} \nabla^2_{ij} w \right) &= v_t - F^{ij} \nabla^2_{ij} v \\
        &= \left( F + \frac{v}{\binom{n}{k}^{\frac{1}{k}}} \right) - F - F^{ij} v \delta_{ij} = \frac{v}{\binom{n}{k}^{\frac{1}{k}}} - v \sum F^{ii} \geqslant 0.
    \end{align*}
    Consider test function $\phi = \frac{F}{w}$, let $\mathcal{L}=\frac{\partial}{\partial t} - F^{ij} \nabla^2_{ij}$, we have
    \begin{align}
        \mathcal{L}\phi &= \frac{\mathcal{L}F}{w} - \frac{F}{w^2} \mathcal{L}w + F^{ij}\frac{F_jw_i}{w^2} + F^{ij}\frac{F_iw_j}{w^2} - \frac{2F F^{ij} w_i w_j}{w^3} \notag \\
        &= F^{ij}\frac{F_jw_i}{w^2} + F^{ij}\frac{F_iw_j}{w^2} - \frac{2F F^{ij} w_i w_j}{w^3} = 2F^{ij}\frac{w_j}{w}\phi_i,
    \end{align}
    where we use the fact that
    \begin{align}
        &\frac{\mathcal{L}F}{w} - \frac{F}{w^2} \mathcal{L}w 
        = \phi \left( - \sum F^{kk} + \frac{1}{\binom{n}{k}^{\frac{1}{k}}} \right) - \frac{F}{w^2} \left( \frac{w}{\binom{n}{k}^{\frac{1}{k}}} - w \sum F^{ii} \right) \notag \\
        =& \phi \left( - \sum F^{kk} + \frac{1}{\binom{n}{k}^{\frac{1}{k}}} \right) - \frac{\phi}{w} \left( \frac{w}{\binom{n}{k}^{\frac{1}{k}}} - w \sum F^{ii} \right) = 0.
    \end{align}
    Hence by maximum principle, $\phi$ attains its minimum and maximum at parabolic boundary. 
    
    On $U_r \times \{0\}$, $F = F(\Lambda_{v_0})$ is uniformly bounded from above and below under the condition \eqref{sigmak lower bound}, together with the discussion in Remark \ref{sigmak upper bound}. 
    On $\partial U_r \times [0, T]$, from equation we have
    \[
    A \chi^{\prime} g = F + \frac{\underline{v}}{\binom{n}{k}^{\frac{1}{k}}}.
    \]
    i.e.
    \[
    \phi = \frac{1}{\binom{n}{k}^{\frac{1}{k}}} - \frac{A \chi^{\prime}}{A \chi + \frac{v_0}{g}} \leqslant \frac{1}{\binom{n}{k}^{\frac{1}{k}}}
    \]
    and 
    \begin{align*}
        \phi = \frac{1}{\binom{n}{k}^{\frac{1}{k}}} - \frac{A \chi^{\prime}}{A \chi + \frac{v_0}{g}} 
        \geqslant \frac{1}{\binom{n}{k}^{\frac{1}{k}}} - \frac{-2Ag}{-v_0}
        \geqslant \frac{1}{2\binom{n}{k}^{\frac{1}{k}}},
    \end{align*}
    if we choose $r \geqslant \tilde{r}$ sufficiently large. 
\end{proof}

\subsection{\texorpdfstring{$C^2$}{C2} boundary estimates}

\begin{lemma}
    Let $v$ be the solution to approximate equation \eqref{approximate normalized support function flow}, then the second order tangential derivatives on the boundary satisfy $|\nabla^2_{\alpha \beta} v| \leqslant C$ on $\partial U_r \times [0, T]$ for $\alpha, \beta < n$.
\end{lemma}
% \begin{proof}
% Let $\tau_\alpha, \tau_\beta$ be some tangential vector fields on $\partial U_r$. For any fixed $t$.
% Since $v - \underline{v} \equiv 0$ on $\partial U_r$ we have  
% \[
% {\nabla}^2_{\alpha \beta}(v - \underline{v}) = -{\nabla}_{\tau_n}(v - \underline{v}) \Pi(\tau_\alpha, \tau_\beta) \text{ on } \partial U_r.
% \]
% where $\tau_n$ is the interior unit normal vector field to $\partial U_r$ and $\Pi$ denotes the second fundamental form of $\partial U_r$. Since we already get $C^1$ estimates, it follows that
% \[
% |\nabla^2_{\alpha \beta}v| \leqslant C \quad \text{on} \quad \partial U_r.
% \]
% \end{proof}
Estimates for pure tangential second derivatives are well known; see, for instance, \cite[Theorem 6.1]{GuanBo}.

In the following, we will consider the linearized operator $\mathscr{L}$ defined by 
\begin{align*}
    \mathscr{L}\phi = \phi_t - F^{ij} \nabla^2_{ij}\phi + \phi \sum_i F^{ii}
\end{align*}
for any smooth function $\phi$. Recall that 
\[
\overline{U}_r = P^{-1}(\overline{B}_r) = \left\{ X \in \mathbb{H}^{n} | x_{n+1} \leqslant \frac{1}{\sqrt{1-r^2}} \right\}.
\]

\begin{lemma}\label{function h}
    Let $v$ be the solution to approximate equation \eqref{approximate normalized support function flow}, then for any positive constant $A$, let
    \[
    h = v - \underline{v} + A \left( \frac{1}{\sqrt{1-r^2}} - x_{n+1} \right),
    \]
    we have
    \[
    \mathscr{L} h \geqslant A \sum_i F^{ii}.
    \]
\end{lemma}
This lemma is the parabolic counterpart of Lemma \ref{auxiliary function for boundary estimates}. The proof is similar and is therefore omitted.
% \begin{proof}
%     First we have
%     \begin{align*}
%         \mathscr{L} \left( \frac{1}{\sqrt{1-r^2}} - x_{n+1} \right) &= \left( \frac{1}{\sqrt{1-r^2}} - x_{n+1} \right) \sum_i F^{ii} + x_{n+1} \sum_i F^{ii} \\
%         &= \frac{1}{\sqrt{1-r^2}} \sum_i F^{ii}.
%     \end{align*}
%     Therefore
%     \begin{align*}
%         \mathscr{L} h &= \left( v - \underline{v} \right)_t - F^{ij} \nabla^2_{ij}\left( v - \underline{v} \right) + \left( v - \underline{v} \right) \sum_i F^{ii} 
%         + A \frac{1}{\sqrt{1-r^2}} \sum_i F^{ii} \\
%         &\geqslant \left( v - \underline{v} \right)_t - \left( F(\Lambda_{v}) - F(\Lambda_{\underline{v}}) \right) 
%         + A \frac{1}{\sqrt{1-r^2}} \sum_i F^{ii} \\
%         &= \frac{v}{\binom{n}{k}^{\frac{1}{k}}} - \underline{v}_t + F(\Lambda_{\underline{v}})
%         + A \frac{1}{\sqrt{1-r^2}} \sum_i F^{ii} \\
%         &\geqslant \frac{v}{\binom{n}{k}^{\frac{1}{k}}} - \frac{\underline{v}}{\binom{n}{k}^{\frac{1}{k}}}
%         + A \frac{1}{\sqrt{1-r^2}} \sum_i F^{ii} \\
%         &\geqslant A \frac{1}{\sqrt{1-r^2}} \sum_i F^{ii} \geqslant A \sum_i F^{ii},
%     \end{align*}
%     where the first inequality is from the concavity of $F$. 
% \end{proof}

Then the estimates of mixed tangential-normal and pure normal second order derivatives are obtained following the same method as in \cite{Entiresigmakcurvatureflow} and \cite{Jiao}. It should also be noted that, despite our equation not fitting the structure in \cite{Jiao}, its proof remains valid and can be applied with some modifications.
% For the sake of completeness, we present the proof here as well.

\begin{lemma}
    Let $v$ be the solution to approximate equation \eqref{approximate normalized support function flow}, then the mixed tangential-normal derivatives on the boundary satisfy $|\nabla^2_{\alpha n} v| \leqslant C$ on $\partial U_r \times [0, T]$ for $\alpha < n$.
\end{lemma}
% \begin{proof}
%     First we differentiate \eqref{approximate normalized support function flow} to get, for each fixed $1 \leqslant k \leqslant n-1$,
%     \begin{align*}
%         (v_t)_k = F^{pq}\Lambda_{pq,k} + \frac{v_k}{\binom{n}{k}^{\frac{1}{k}}}.
%     \end{align*}
%     Therefore
%     \begin{align*}
%         &\left| \mathscr{L} \nabla_k \left( v-\underline{v} \right) \right| \\
%         =& \left| (v_k)_t - F^{ij} \nabla^2_{ij}v_k + v_k \sum F^{ii} - \mathscr{L} \underline{v}_k \right| \\
%         \leqslant& C + C \sum F^{ii} \leqslant C \sum F^{ii}.
%     \end{align*}
%     Then Lemma \ref{auxiliary Lemma} applies to $\pm \nabla_k \left( v-\underline{v} \right)$ to yield
%     \begin{align*}
%         |\nabla^2_{\alpha n} v| \leqslant C.
%     \end{align*}
% \end{proof}
\begin{remark}
    Since $\chi$ is bounded, the upper bound is independent of time $t$.
\end{remark}

\begin{lemma}[Lemma 10 in \cite{Entiresigmakcurvatureflow}]
    Let $v$ be the solution to approximate equation \eqref{approximate normalized support function flow}, then the second order normal derivatives on the boundary satisfy $|\nabla^2_{nn} v| \leqslant C$ on $\partial U_r \times [0, T]$.
\end{lemma}
\begin{proof}
    By Lemma 10 in \cite{Entiresigmakcurvatureflow}, the lemma can be proved similarly.
    
    Indeed, from equation \eqref{approximate normalized support function flow}, $v = \underline{v}$ on $\partial U_r \times [0, T]$, thus on $\partial U_r \times [0, T]$,
    \[
    v_t = A\chi^{\prime}g,
    \]
    and $g$ is a constant on $\partial U_r$. Hence on $\partial U_r \times [0, T]$,
    \[
    F = -\frac{v_0}{\binom{n}{k}^{\frac{1}{k}}} + A\chi^{\prime}g.
    \]
    This is equivalent to
    \[
    \frac{\sigma_n}{\sigma_{n-k}} = \left( -\frac{v_0}{\binom{n}{k}^{\frac{1}{k}}} + A\chi^{\prime}g \right)^k.
    \]
    When $r$ is large, the right-hand side is strictly positive. Then we can take the same notations as Lemma 10 in \cite{Entiresigmakcurvatureflow}. We still let
    \[
    d(x, t) \coloneqq f - \left( -\frac{v_0}{\binom{n}{k}^{\frac{1}{k}}} + A\chi^{\prime}g \right),
    \]
    where $f = \left( \frac{\sigma_{n-1}(\kappa^{\prime})}{\sigma_{n-k-1}(\kappa^{\prime})} \right)^{\frac{1}{k}}$.
    From the subsolution condition at infinity and the choice of $r$, we still have $d(x, 0) \geqslant c > 0$ on $\partial U_r \times \{0\}$.
    % We adopt the same argument as Lemma 10 in \cite{Entiresigmakcurvatureflow}. 
    % For the sake of completeness, we present the details adapted to our situation. 
    Therefore, we may assume that $d(x, t)$ achieves its minimum at the point $(y, t_0) \in \partial U_r \times (0, T]$. 
    
    Then as in \cite{Entiresigmakcurvatureflow}, the same argument goes through if we replace 
    \[[(1+\alpha)\tilde{t}]^{\frac{1}{1+\alpha}} \left( \frac{\sqrt{1-r^2}}{-u_0^*} \right)^{\frac{1}{\alpha}}\] 
    by
    \[ -\frac{v_0}{\binom{n}{k}^{\frac{1}{k}}} + A\chi^{\prime}g. \]

    Recall that $\underline{v} = A\chi g + v_0$, hence the covariant derivatives of $\underline{v}$ on boundary are bounded, with constants independent of time, and so are those of $\psi$. In a word, we conclude the double normal estimates independent of time $t$.
\end{proof}

\subsection{\texorpdfstring{$C^2$}{C2} global estimates}

\begin{lemma}[Lemma 20 in \cite{Entireconvexcurvatureflow}]
    Let $v$ be the solution to approximate equation \eqref{approximate normalized support function flow}. Denote the eigenvalues of $\left[ g^{-1} \circ \left( \nabla^2 v - vg \right) \right]$ by $(\lambda_1, \cdots, \lambda_n)$ with $\lambda_1 \geqslant \lambda_2 \geqslant \cdots \geqslant \lambda_n > 0$. Then $\lambda_1$ is bounded from above on $\overline{U}_r \times [0, T]$. 
\end{lemma}
\begin{proof}
    The proof is almost the same as Lemma 20 in \cite{Entireconvexcurvatureflow}, except that here we consider the normalized flow with an additional term.
\end{proof}

From the above discussion, we can see that the approximate problem \eqref{approximate normalized support function flow} has a unique solution $v_r$ for any $r$ sufficiently close to $1$ and any $T>0$.

\section{Convergence to entire solution: inverse Gauss curvature flow}\label{Convergence to entire solution: inverse Gauss curvature flow}

In this section, we want to show that there exists a subsequence of  $\left\{ u_r^*(\xi, t)\right\}$ defined by $v_r(P^{-1}(\xi), t)\sqrt{1-|\xi|^2}$ that converges to the entire solution for the inverse Gauss curvature flow, i.e. $k=n$ case in \eqref{normalized support function flow}. 
The evolution equation for $u^*_{r}$ is 
\begin{align}
    \begin{cases}
    (u^*_r)_t = F(w^* \gamma_{ik}^* (u^*_r)_{kl} \gamma_{lj}^*)w^* + \frac{u^*_r}{\binom{n}{k}^{\frac{1}{k}}} & \text{in } B_r \times (0, T], \\[4pt]
    u^*_r(\cdot, t) = \underline{u}^* & \text{on } \partial B_r \times [0, T], \\[4pt]
    u^*_r(\cdot, 0) = \underline{u}^* = u^*_0 & \text{on } B_r \times \{0\}. 
    \end{cases}
    \label{normalized support function flow in B1 model: inverse Gauss curvature}
\end{align}
In this section, we set $F=\sigma_n^{\frac{1}{n}}$.

\subsection{Local \texorpdfstring{$C^0$}{C0} estimates}

Since $\overline{v}$ and $\underline{v}$ are upper and lower barrier of equation \eqref{approximate normalized support function flow}, we have
\[
\underline{v} \leqslant v_{r} \leqslant \overline{v} \quad \text{in } \overline{U}_r \times [0, T].
\]
Hence
\[
\underline{u}^* \leqslant u^*_{r} \leqslant \overline{u}^* \quad \text{in } \overline{B}_r \times [0, T].
\]
Since $\underline{v} = A\chi(t)g + v_0$ with $\chi(t) \equiv 1$ for $t \geqslant 1$, and $\overline{v}$ is uniformly bounded by maximum principle, we obtain local $C^0$ estimates independent of $t$.

\subsection{Local \texorpdfstring{$C^1$}{C1} estimates}

By the same argument as in section \ref{Gauss Local C1}, we get local $C^1$ estimates.

\subsection{Local estimates for \texorpdfstring{$F$}{F}}

With condition \eqref{sigmak lower bound}, Lemma \ref{Upper and lower bounds for F} provides uniform upper and lower bounds for the speed function $F$.

\subsection{Local \texorpdfstring{$C^2$}{C2} estimates}

\begin{lemma}[Lemma 28 in \cite{Entireconvexcurvatureflow}]
    Let $u^*_r$ be the solution of \eqref{normalized support function flow in B1 model: inverse Gauss curvature}. For $l(\xi)$ any linear function, consider the set $\left\{ l(\xi) - u^*_r(\xi, t) > 0 \right\}$, then we have that, in this set
    \[
    \left( l - u^*_r \right)^{\beta}\lambda_{\max} \leqslant C
    \]
    for $t \in [0, T]$, where $\beta$ is a large positive constant and $C$ is independent of $r$ and $t$.
\end{lemma}
\begin{proof}
    The proof is the same as Lemma 28 in \cite{Entireconvexcurvatureflow}. Recall here $\tilde{G} = F$ and $\alpha = -1$, the evolution equation for support function is
    \[
    v_t = F + \gamma v.
    \]
    Since we already have local $C^0$ and $C^1$ estimates which are uniform in $t$, the derivation shows that $C$ does not depend on $t$.
\end{proof}

\begin{remark}
    Since the Hessian quotient operator $F$ for $1 \leqslant k <n$ does not satisfy condition (1.6) in \cite{Entireconvexcurvatureflow}, we can only obtain this lemma for the case where $k=n$.
\end{remark}

By above discussion, we obtain the local $C^2$ estimates for $u^*_r$ in $K \times [0, T]$ for any compact set $K \subseteq B_1$ and $T>0$. Therefore we can find a subsequence of $\left\{ u^*_r \right\}$ such that $u^*_r \rightarrow u^*$ in $[0, T]$. Since $T$ is arbitrary, we obtain the existence of an entire solution in $[0, \infty)$ for the inverse Gauss curvature flow.

\section{Convergence to entire solution: inverse Hessian curvature flow}\label{Convergence to entire solution: inverse Hessian curvature flow}

As in Section \ref{Convergence to entire solution: inverse Gauss curvature flow}, the approximate problem in $B_1$ model is \eqref{normalized support function flow in B1 model: inverse Gauss curvature} with $F=\left(\frac{\sigma_n}{\sigma_{n-k}}\right)^{1/k}$. Denote its solution by $u^*_r$.
Let $u_r$ be the Legendre transform of $u^*_r$, then it satisfies
\begin{align}
    (u_r)_t = G \left( \frac{1}{w} \gamma^{ik} (u_r)_{kl} \gamma^{lj} \right) w - \gamma \left( x \cdot Du_r - u_r \right) \text{ in } \Omega_r(t) \times (0, T], \label{normalized approximate flow by Legendre transform}
\end{align}
where $G = -\sigma_k^{-\frac{1}{k}}<0$, $w=\sqrt{1-|Du|^2}$ and $\Omega_r(t) \coloneqq Du_r^*(B_r, t)$. 
In the following, we want to show that there exists a subsequence of $\left\{ u_r \right\}$ that converges locally uniformly to an entire solution. We will derive local estimates for $u_r$ with the same methods as in \cite{Entiresigmakcurvatureflow}.

\subsection{Local \texorpdfstring{$C^0$}{C0} estimates}

Since negative constants are solutions to normalized flow \eqref{approximate normalized support function flow}, by asymptotic maximum principle, we have
\[
-c_2 \leqslant v_r \leqslant -c_1.
\]
Therefore let $u_r$ be the Legendre transform of $u^*_r$, we have, for $(x,t)\in\Omega_r(t)$, 
\[
\sqrt{c_1^2 + |x|^2} \leqslant u_r(x,t) \leqslant \sqrt{c_2^2 + |x|^2}.
\]
For convenience, denote
\[
\underline{u} = \sqrt{c_1^2 + |x|^2},
\text{ and } 
\overline{u} = \sqrt{c_2^2 + |x|^2}.
\]

\subsection{Local \texorpdfstring{$C^1$}{C1} estimates}

First choose $0 < c_0 < c_1$, and let $\hat{\underline{u}} = \sqrt{c_0^2 + |x|^2}$, then we have
\[
\hat{\underline{u}} < \underline{u} \leqslant u_r \leqslant \overline{u}.
\]
For any compact set $K$ and any positive constant $T$, let
\[
2\delta = \min_{K \times [0, T]} \left( \underline{u}-\hat{\underline{u}} \right).
\]
We can see $\delta$ is independent of $t$. Then define $\psi = \hat{\underline{u}} + \delta$. Denote $\Omega = \{(x,t) \in \mathbb{R}^n \times [0,T] \, | \, \psi \leqslant \overline{u}\}$. Then 
\[
K \times [0, T] \subseteq \Omega.
\]
Since all these $\hat{\underline{u}}, \underline{u}, \overline{u} \rightarrow \left| x \right|$ as $\left| x \right| \to +\infty$, we know that $\Omega$ is a compact set. Denote $\Omega_r(t) = Du^*_r(B_r, t)$. Then by Lemma \ref{Local C1}, if $\Omega \subseteq \bigcup_{t \in [0,T]} \Omega_r(t)$, we have the gradient estimate:
\[
\sup_{K \times [0,T]} \frac{1}{\sqrt{1 - |Du_r|^2}} \leqslant \frac{1}{\delta} \sup_{\Omega} \frac{\overline{u} - \psi}{\sqrt{1 - |D\psi|^2}}.
\]

% \subsection{Some calculation}

\subsection{Local estimates for \texorpdfstring{$G$}{G}}

With condition \eqref{sigmak lower bound}, Lemma \ref{Upper and lower bounds for F} provides uniform upper and lower bounds for the speed function $-G = F$.

\subsection{Local \texorpdfstring{$C^2$}{C2} estimates}

Let $\mathfrak{L} = \partial_t - G^{ij}\nabla^2_{ij}$, where $\nabla$ denotes the connection on hypersurface. Recall $u = -\langle \tilde{X}, E_{n+1} \rangle$ and let $\phi = -\langle \nu, E_{n+1} \rangle$ for parametrization with respect to the initial hypersurface, where $\tilde{X}$ is the normalized hypersurface defined by height function $u$, then the normalized flow is 
\begin{align}\label{normalized flow}
    \tilde{X}_t = G \nu + \gamma \tilde{X}.
\end{align}
Then by same calculation as Lemma 11 in \cite{Entiresigmakcurvatureflow}, we can derive
\begin{lemma}\label{evolution for some geometric quantity}
    Under the normalized flow \eqref{normalized flow}, we have
    \begin{align*}
        \mathfrak{L} u &= 2G\phi + \gamma u, \\
        % \nu_t &= \nabla G \\
        \mathfrak{L} \phi &= -\sum_i G^{ii}\kappa_i^2 \phi, \\
        \mathfrak{L} G &= -\sum_i G^{ii}\kappa_i^2 G + \gamma G, \\
        \mathfrak{L} h_{i}^{j} &= -2G \sum_k h_{i}^{k} h_{k}^{j} - h_{i}^{j} \sum_{k,l,s} G^{kl} h_k^s h_{sl} + \sum_{p,q,r,s} G^{pq,rs} \nabla_i h_{pq} \nabla^j h_{rs} - \gamma h_{i}^{j}.
        % \mathfrak{L} v &= -\sum_i G^{ii}\kappa_i^2 v + \gamma v.
    \end{align*}
    % Here $v = \langle X, \nu \rangle$ is the support function.
\end{lemma}

\begin{lemma}[Lemma 13 in \cite{Entiresigmakcurvatureflow}]
Let $u_r^*$ be the solution of \eqref{normalized support function flow in B1 model: inverse Gauss curvature} with $F=\left(\frac{\sigma_n}{\sigma_{n-k}}\right)^{1/k}$. Let $u_r$ be the Legendre transform of $u_r^*$. Denote $\Omega_r(t) \coloneqq Du_r^*(B_r, t)$. For any given $c > 0$, let $r_c \in (0, 1)$ such that when $r > r_c$, $u_r(\cdot, t)|_{\partial\Omega_r(t)} > c$ for all $t \in [0, T]$. Then for $r > r_c$ we have
\[
(c - u_r)^m \log \kappa_{\max}(x, t) \leqslant C,
\]
where $\kappa_{\max}(x, t)$ is the largest principal curvature of $\mathcal{M}_{u_r}$ at $(x, t)$, $m$ is a large constant only depending on $k$, and $C \coloneqq C(G, c) > 0$ is independent of $r$ and time $t$.
\end{lemma}
\begin{proof}
    The proof follows the same structure as Lemma 13 in \cite{Entiresigmakcurvatureflow}, except that we use Lemma \ref{evolution for some geometric quantity} instead and take $\alpha = -1$.
\end{proof}

By the above discussion, we obtain the local $C^2$ estimates for $u_r$ in $K \times [0, T]$ for any compact set $K \subseteq \mathbb{R}^n$ and $T>0$. Therefore we can find a subsequence of $\left\{ u_r \right\}$ such that $u_r \rightarrow u$ in $[0, T]$. Since $T$ is arbitrary, we obtain the existence of an entire solution in $[0, \infty)$ for general inverse Hessian curvature flow.

\subsection{Barriers}\label{Barriers}

Same as the argument as in section \ref{Proof of Theorem 1.1}, we take the solution to inverse Gauss curvature flow as new supersolution, still denote its support function by $\overline{v}$. By comparison principle we still have
\[
\underline{v} \leqslant v_r \leqslant \overline{v}.
\]
Then define $\underline{u}^*(\xi, t) = \underline{v}(y(\xi), t) \sqrt{1 - |\xi |^2}$ and $\overline{u}^*(\xi, t) = \overline{v}(y(\xi), t) \sqrt{1 - |\xi |^2}$, we have
\begin{align*}
    \underline{u}^* \leqslant u_r^* \leqslant \overline{u}^*.
\end{align*}
Let $\underline{u}$, $u_r$ and $\overline{u}$ be the Legendre transform of $\overline{u}^*$ , $u_r^*$ and $\underline{u}^*$ respectively. Then we have
\begin{align*}
    \underline{u} \leqslant u_r \leqslant \overline{u}.
\end{align*}

\section{Convergence to the self-shrinker}

We now prove the convergence of the normalized flow. By the estimates obtained in the previous two sections, for every $K \subset \subset \mathbb R^n$, we have 
\begin{align}
    \sup_{t \geqslant 0}\|u_r(\cdot, t)\|_{C^2(K)}
    \leqslant C_{K},
\end{align}
where the constant $C_{K}$ is independent of $r$ and $t$. 
Furthermore, since $u_r$ satisfies \eqref{normalized approximate flow by Legendre transform}, by the local uniform parabolicity of the equation, together with the parabolic Evans--Krylov theorem and the Schauder estimates, for every $K \subset \subset \mathbb{R}^n$, every $\tau>0$, and every integer $m \geqslant 0$, there exists a constant $C_{K,m,\tau}$, independent of $r$ and $t$, such that
\[ 
\sup_{t \geqslant \tau}\|u_r\|_{C^m(K\times[t,t+1])}
\leqslant C_{K,m,\tau},
\]

Hence, letting $r \to 1$ and using a diagonal argument, we obtain an entire solution
$u$ of the normalized flow on
\[
\mathbb R^n\times[0,+\infty).
\]
Moreover, since the local estimates for $u_r$ are uniform in $r$, we have
% for every $K \subset \subset \mathbb{R}^n$,
% \begin{align}\label{local estimates for u}
%     \sup_{t \geqslant 0}\|u(\cdot,t)\|_{C^2(K)}
%     \leqslant C_{K}.
% \end{align}
\begin{align}\label{local estimates for u}
    \sup_{t \geqslant \tau}\|u\|_{C^m(K\times[t,t+1])}
    \leqslant C_{K,m,\tau}.
\end{align}

By the discussion in Section \ref{Barriers}, we know that 
\begin{align}\label{barriers}
    \underline{u} \leqslant u \leqslant \overline{u}.
\end{align}
Let $v$ be the corresponding support function of $u$ defined on $\mathbb H^n$. Then
\[
v_t = F(\Lambda_v)+\gamma v, \quad \gamma = \frac{1}{\binom{n}{k}^{\frac{1}{k}}},
\]
and the boundary value at infinity holds,
$
v(\cdot,t) = \varphi$ on  $\mathbb H^n(\infty).
$
% Let $v_\infty$ be the unique solution of \eqref{sigmak selfshrinker}.
% Set
% \[
% z=v-v_\infty.
% \]
By \eqref{normalized approximate flow by Legendre transform}, we can define
\begin{align}
    q \coloneqq \frac{u_t}{w} = G - \gamma v = -v_t.
\end{align}

\begin{lemma}\label{q+ uniform}
    We first show that $q_+(x, t) = \max\left\{0,q(x,t)\right\}$ tends to zero at infinity uniformly with respect to time, i.e.
    \begin{align}
    \lim_{R\to+\infty} \sup_{\substack{|x| \geqslant R \\ t \geqslant 0}}q_+(x, t)=0.
    \end{align}
\end{lemma}
\begin{proof}
    Let $v_r$ be the solution of \eqref{approximate normalized support function flow}. Define
    \[
    H_r \coloneqq (v_r)_t = F(\Lambda_{v_r}) + \gamma v_r.
    \]
    Then $H_r$ satisfies
    \begin{align}
    (H_r)_t = F^{ij} (H_r)_{ij} - c H_r,
    \end{align}
    where $c = \sum F^{ii} - \gamma \geqslant 0$. Set
    $
    L = \partial_t - F^{ij} \nabla^2_{ij} + c,
    $
    and we have $L H_r = 0$.
    Let
    \[
    g(\rho) = -(\cosh\rho)^{-a},
    \quad
    0 < a < \frac{n-1}{n},
    \]
    be the function constructed in Lemma \ref{smooth function g}, which satisfies
    \[
    \Lambda_g > 0, \quad F(\Lambda_g) + \gamma g > 0.
    \]
    By the concavity and homogeneity of $F$, we have
    \[
    F(B)\leqslant F^{ij}(A)B_{ij}
    \]
    for any admissible $A$ and $B$. Therefore, we get
    \begin{align*}
    Lg &= - F^{ij}\left( g_{ij} - g\delta_{ij} \right) - \gamma g \leqslant - F(\Lambda_g) - \gamma g < 0.
    \end{align*}
    By the subsolution condition at infinity \eqref{strict subsolution}, we can choose $C_0$ large such that
    \[
    H_r(\cdot, 0) - C_0 g \geqslant 0 \quad \text{in } U_r.
    \]
    On $\partial U_r$, we have
    \[
    H_r = A \chi^{\prime}(t) g \geqslant 2A g.
    \]
    Thus
    \[
    H_r - C_0 g \geqslant (2A - C_0)g \geqslant 0,
    \]
    if we choose $C_0 \geqslant 2A$. Let $W_r = H_r - C_0 g$, then
    \[
    L W_r = L H_r - C_0 Lg > 0.
    \]
    By the parabolic maximum principle, we have
    \[
    H_r \geqslant C_0 g.
    \]
    Hence
    \[
    -H_r \leqslant C_0 (\cosh\rho)^{-a}.
    \]
    Rewriting this in $\mathbb{R}^n$, we have
    \[
    q_r(x, t)=-(v_r)_t(Du_r(x,t),t)\leqslant C_0 (1-|Du_r(x,t)|^2)^{a/2}.
    \]
    Then taking the limit $r \to 1$ yields 
    \[
    q(x, t) \leqslant C_0 (1-|Du(x,t)|^2)^{a/2}=C_0w^a(x,t).
    \]
    By the local $C^0$ estimates, we already know that
    \begin{align}\label{hyperboloid trapping}
        \sqrt{c_1^2 + |x|^2} \leqslant u_r(x,t) \leqslant \sqrt{c_2^2 + |x|^2}.
    \end{align}
    Hence $w(x, t)$ converges to $0$ as $|x| \to +\infty$ uniformly with respect to time, which completes the proof.
\end{proof}

\begin{proof}[Proof of Theorem \ref{inverse Hessian curvature flow}]
    By using Lemma \ref{evolution for some geometric quantity}, we can directly calculate that 
    \begin{align}\label{equaiton of q}
    q_t-a^{ij}q_{ij}-b^{i}q_{i}+cq=0,
    \end{align}
    where
    \[
    c 
    % = \frac{w_t - a^{ij}w_{ij} - \beta^{i}w_i - \gamma w}{w}
    % = a^{ij}\left[\frac{u_{ki}u_{kj}}{w^2}+\frac{(u_k u_{ki})(u_\ell u_{\ell j})}{w^4}\right]-\gamma
    % = a^{ij}g^{k\ell}h_{ki}h_{\ell j}-\gamma
    = \sum G^{ii}\kappa_i^2 - \gamma \geqslant 0.
    \]
    Define
    \begin{align}
        M(t) = \max\left\{0,\sup_{\mathbb R^n}q(\cdot,t)\right\}.
    \end{align}
    Since $q$ satisfies \eqref{equaiton of q}, the maximum principle and the previous lemma imply that $M$ is nonincreasing. Hence
    \[
    \ell \coloneqq \lim_{t\to+\infty}M(t)
    \]
    exists. We claim that $\ell=0$. Suppose, to the contrary, that $\ell>0$. Choose $t_j\to+\infty$ and $x_j$ such that
    \[
    q(x_j,t_j)=M(t_j)\to\ell.
    \]
    The uniform decay at infinity of $q_+$ implies that the points $x_j$ remain in a fixed compact subset of $\mathbb R^n$. After passing to a subsequence, we can assume 
    \[
    x_j\to x_\infty.
    \]
    Consider the time translations
    \[
    u_j(x,s) \coloneqq u(x,t_j+s).
    \]
    By \eqref{local estimates for u}, after passing to a subsequence,
    \begin{align}
    u_j \rightarrow U \quad \text{in } C^\infty_{\mathrm{loc}}(\mathbb R^n \times \mathbb R).
    \end{align}
    Set
    \[
    W(x,s)\coloneqq \sqrt{1-|DU(x,s)|^2}, \quad q_\infty(x,s)\coloneqq \frac{U_s(x,s)}{W(x,s)}.
    \]
    Since the convergence $u_j\to U$ is locally smooth, we also have
    \[
    q(x,t_j+s)=\frac{(u_j)_s(x,s)}{\sqrt{1-|Du_j(x,s)|^2}}\longrightarrow q_\infty(x,s)
    \]
    locally smoothly on $\mathbb R^n\times\mathbb R$. For every fixed $s\in\mathbb R$, we have $q(x,t_j+s)\leqslant M(t_j+s)$. Since $t_j\to+\infty$, we have
    \[
    q_\infty(x,s)\leqslant \ell \quad \text{on } \mathbb R^n\times\mathbb R.
    \]
    Moreover, since $x_j\to x_\infty$ and $q(x_j,t_j)=M(t_j)\to\ell$, the local smooth convergence gives
    \[
    q_\infty(x_\infty,0)=\ell.
    \]
    Recall that $q$ satisfies \eqref{equaiton of q}. Passing to the limit along the sequence $u_j$, we obtain
    \[
    (q_\infty)_s-a_\infty^{ij}(q_\infty)_{ij}-b_\infty^i(q_\infty)_i+c_\infty q_\infty=0, \quad c_\infty\geqslant 0.
    \]
    Let $Z\coloneqq \ell-q_\infty$. Then $Z\geqslant 0$, $Z(x_\infty,0)=0$, and
    \[
    Z_s-a_\infty^{ij}Z_{ij}-b_\infty^iZ_i+c_\infty Z=c_\infty\ell\geqslant 0.
    \]
    By the strong parabolic minimum principle,
    % \[
    % Z\equiv0 \quad \text{on } \mathbb R^n\times(-\infty,0].
    % \]
    % Indeed, the strong minimum principle may first be applied on a bounded parabolic cylinder containing $(x_\infty,0)$ as an interior point, and the equality can then be propagated to the whole connected backward parabolic region by applying the same argument on overlapping cylinders. Consequently,
    \[
    q_\infty\equiv\ell \quad \text{on } \mathbb R^n\times(-\infty,0],
    \]
    which contradicts to Lemma \ref{q+ uniform} whenever $\ell>0$. Hence $\ell=0$, and consequently
    \[
    \lim_{t\to+\infty}M(t)=0.
    \]
    % Since $q_\infty=U_s/W$, we obtain $U_s=\ell W$ for all $s\leqslant 0$. By the local gradient estimate, applied on a fixed compact neighborhood of the origin and then passed to the limit, there exists a constant $\delta>0$ such that
    % \[
    % W(0,s)\geqslant \delta \quad \text{for all } s\leqslant 0.
    % \]
    % Hence, for every $T>0$,
    % \[
    % U(0,0)-U(0,-T)=\int_{-T}^0 U_s(0,s)\,ds
    % =\ell\int_{-T}^0 W(0,s)\,ds
    % \geqslant \ell\delta T.
    % \]
    % On the other hand, the uniform hyperboloid trapping for $u$ passes to the limit and gives
    % \[
    % c_1\leqslant U(0,s)\leqslant c_2 \quad \text{for all } s\in\mathbb R.
    % \]
    % Therefore
    % \[
    % U(0,0)-U(0,-T)\leqslant c_2-c_1,
    % \]
    % which contradicts the preceding inequality as $T\to+\infty$ whenever $\ell>0$. Hence $\ell=0$, and consequently
    % \[
    % \lim_{t\to+\infty}M(t)=0.
    % \]

    Now let $t_j\to+\infty$ be arbitrary. By the local estimates, after passing to a subsequence,
    \[
    u(x,t_j+s)\longrightarrow U(x,s).
    \]
    For every fixed $(x,s)$,
    \[
    \frac{U_s(x,s)}{W(x,s)} = \lim_{j\to\infty}q(x,t_j+s) \leqslant \lim_{j\to\infty}M(t_j+s)=0.
    \]
    Therefore, since $W>0$, we have
    \begin{equation}
    U_s \leqslant 0.
    \end{equation}
    Thus $U(x,s)$ is nonincreasing in $s$ for each fixed $x$. The hyperboloid trapping \eqref{hyperboloid trapping} gives a time-independent lower and upper bound, so the pointwise limits
    \begin{equation}
    U^-(x) \coloneqq \lim_{s\to-\infty}U(x,s), \quad U^+(x) \coloneqq \lim_{s\to+\infty}U(x,s)
    \end{equation}
    exist, with
    \begin{equation}
    U^- \geqslant U(\cdot,s) \geqslant U^+.
    \end{equation}
    We show that $U^+$ is stationary. Let $s_j\to+\infty$. The local estimates allow us to extract a subsequence such that
    \[
    U_j(x,\tau) \coloneqq U(x,s_j+\tau)
    \]
    converges locally smoothly on $\mathbb R^n\times\mathbb R$ to some $\hat U(x,\tau)$. But for every fixed $(x,\tau)$,
    \[
    U(x,s_j+\tau)\longrightarrow U^+(x)
    \]
    by monotonicity. Hence
    \[
    \hat U(x,\tau)=U^+(x),
    \]
    which is independent of $\tau$. Therefore $U^+$ is a stationary solution of \eqref{normalized approximate flow by Legendre transform}. The same argument with $s_j\to-\infty$ shows that $U^-$ is a stationary solution too.

    It remains to identify the boundary values at infinity of the stationary limits $U^-$ and $U^+$. Let $V^-$ and $V^+$ denote their corresponding support functions. Since along the normalized flow we have
    \[
    \underline v(\cdot,t)\leqslant v(\cdot,t)\leqslant \overline v(\cdot,t),
    \]
    we may pass to the limit along the sequences defining $U^\pm$. For $t \geqslant 1$, the lower barrier is independent of time and is given by
    \[
    \underline v=v_0+Ag,
    \]
    whereas the upper barrier satisfies
    \[
    \overline v(\cdot,t) \rightarrow h
    \]
    as $t\to+\infty$, where $h$ is the harmonic extension of $\varphi$. Consequently,
    \[
    v_0+Ag\leqslant V^\pm\leqslant h.
    \]
    Since
    \[
    v_0+Ag=h=\varphi \quad \text{on }\mathbb H^n(\infty),
    \]
    both $V^-$ and $V^+$ have boundary value $\varphi$ at infinity. Since $U^-$ and $U^+$ are stationary solutions, the uniqueness of the self-shrinker with prescribed boundary value $\varphi$ yields
    \[
    U^-=U^+=u_\infty.
    \]

    Since every sequence $t_j\to\infty$ has a subsequence whose time translations converge to $u_\infty$, and the limit is unique, the entire family converges:
    \begin{equation}
    u(\cdot,t)\longrightarrow u_\infty \quad \text{in }C^\infty_{\mathrm{loc}}(\mathbb{R}^n) \quad \text{as }t\to\infty.
    \end{equation}
    Equivalently, the normalized hypersurfaces converge locally smoothly to the unique self-shrinker whose support function has boundary value $\varphi$ at infinity.
\end{proof}

\section{Discussion about the initial hypersurface}

In this section, we give a large collection of initial hypersurfaces to satisfy conditions in Theorem  \ref{inverse Hessian curvature flow}. Roughly speaking, on sufficiently large compact subsets in the interior of the initial hypersurface, it can be prescribed arbitrarily.

%These examples satisfy \eqref{sigmak lower bound} in the following two cases: (i)$k=1$ for arbitrary boundary value $\varphi$; (ii)$\varphi \equiv -1$ for $1 \leqslant k \leqslant n-1$.

Let $v$ be the unique self-shrinker of equation \eqref{sigmak selfshrinker} by Theorem \ref{self-shrinker existence} for given $\varphi$. Denote $g$ as in Lemma \ref{smooth function g} for some $0 < a < \frac{n-1}{n}$. Hence $g$ is a strict subsolution. Consider 
\begin{equation}\label{v_0}
v_0 =  v + \mu g
\end{equation}
where $\mu$ is any positive constant. Then by concavity, $v_0$ is an entire strict subsolution, and
\[
v_0 = \varphi \quad \text{on } \mathbb{H}^n(\infty).
\]

Since the strict subsolution inequality is required only near infinity, we may enlarge the exceptional compact set. More precisely, using a regularized-maximum gluing argument, we can replace $v_0$ on a compact set by a smooth negative strictly hyperbolic convex function, while leaving $v_0$ unchanged outside a slightly larger compact set. The resulting function is smooth, negative, and hyperbolic convex on $\mathbb{H}^n$, and it  satisfies \eqref{strict subsolution} outside the enlarged compact set. We need to use the following regularized-maximum function lemma (see Lemma 5.18 in \cite{demailly1997complex}):

\begin{lemma}\label{gluing lemma}
    Let $\eta>0$. There exists a function $M_{\eta} \in C^{\infty}(\mathbb{R}^2)$ with 
    \begin{enumerate}[label=\text{(\alph*)}]
        \item $M_{\eta}(t_1, t_2)$ is nondecreasing with respect to each variable, smooth and convex on $\mathbb{R}^2$;
        \item $\max\{t_1, t_2\} \leqslant M_{\eta}(t_1, t_2) \leqslant \max\{t_1+\eta, t_2+\eta\}$;
        \item $M_{\eta}(t_1, t_2) = \max\{t_1, t_2\}$, if $|t_1 - t_2| \geqslant 2\eta$;
        \item $M_{\eta}(t_1 + a, t_2 + a) = M_{\eta}(t_1, t_2) + a$ for any constant $a$.  
    \end{enumerate}
\end{lemma}

\begin{proposition}
    Let $M_\eta$ as above lemma. If $u_1, u_2$ are two strictly convex functions, $M_{\eta}(u_1, u_2)$ is also strictly convex.
\end{proposition}
\begin{proof}
    Let $z = (1-\lambda) x + \lambda y$ for some $\lambda \in (0, 1)$. Since $u_1, u_2$ are strictly convex, we have
    \[
    u_1(z) < (1-\lambda) u_1(x) + \lambda u_1(y) \eqqcolon A,
    \]
    and
    \[
    u_2(z) < (1-\lambda) u_2(x) + \lambda u_2(y) \eqqcolon B.
    \]
    Define
    \[
    \delta \coloneqq \min \left\{ A - u_1(z), B - u_2(z) \right\} > 0,
    \]
    and 
    \[
    w(x) \coloneqq M_{\eta}(u_1(x), u_2(x)).
    \]
    Then since $M_\eta$ is nondecreasing, 
    \[
    M_{\eta}(u_1(z), u_2(z)) \leqslant M_{\eta}(A-\delta, B-\delta) = M_{\eta}(A, B)-\delta.
    \]
    by property $(d)$ of Lemma \ref{gluing lemma}. Hence
    \begin{align*}
        w(z) \leqslant M_{\eta}(A, B)-\delta \leqslant (1-\lambda) w(x) + \lambda w(y) - \delta < (1-\lambda) w(x) + \lambda w(y),
    \end{align*}
    which completes the proof.
\end{proof}

Let $u^*_0(\xi) = v_0(P^{-1}(\xi))w^*(\xi)$ defined in $B_1^{\xi}$ as before. Then $u^*_0$ is negative and strictly convex. 

\begin{lemma}\label{13.3}
    Given any smooth strictly convex function $h<0$ defined in some compact strictly convex set $K \subset \subset B_1$. For convenience, after subtracting a linear function, we may assume that the minimum point of $h$ lies in the interior of $K$. Then there exists some nonnegative constant $c_*(h) \geqslant 0$, such that if $c \geqslant c_*$, $h - c$ can be glued with $u_0^*$ to obtain a strictly convex function in $B_1$ satisfying \eqref{strict subsolution}.
\end{lemma}
\begin{proof}
    First, let $d$ denote the signed distance function from the boundary of $K$, then $d$ is a convex function. For $\delta>0$ small, define the hypersurface $\left\{ d = -\delta \right\}$ by $K_1$, the boundary of $K$ by $K_2$. We now consider the ring-shaped area $\Omega$ enclosed by $K_1$ and $K_2$. 
    Let
    \[
    A_1 \coloneqq \min_{K_1} h, \quad A_2 \coloneqq \max_{K_2} h.
    \]
    Then choose a smooth strictly convex increasing function $f$ defined on $[-\delta, l]$ such that
    \[
    f(-\delta) \leqslant A_1 - 2\eta, \quad f(0) \geqslant A_2 + 2\eta, 
    % \quad f(l) < 0,
    \]
    for some small positive constant $\eta>0$, where $l$ is the maximal existence time, i.e. $l = \max\left\{ d(\xi, K) | \xi \in \partial B_1 \right\}$. For example, $f$ can be chosen as
    \[ 
    f(t) \coloneqq A_2 + 2\eta + \left(\frac{A_2 - A_1 + 4\eta}{\delta} + \varepsilon\delta\right)t + \varepsilon t^2, \quad t \in [-\delta,l], 
    \]
    where $0< \varepsilon < \frac{A_2 - A_1 + 4\eta}{\delta^2}$. However, $f(l)$ may be positive. If $f(l) < 0$, choose $c_* = 0$ and $\eta$ such that $f(l) + 2\eta < 0$; otherwise $f(l)\geqslant 0$, let $c_* = f(l) + 4\eta$.

    For our convenience, in the following, we will use $h, f(d)$ to denote the functions $h-c, f(d)-c$ translating downward by $c$, where $c \geqslant c_*$ is a constant.
    
    By direct calculation, we have 
    \begin{align*}
        D^2 f(d) = f^{\prime}(d) D^2 d + f^{\prime \prime}(d) D d \otimes D d > 0.
    \end{align*}
    Here $D d \otimes D d$ is positive and $D^2 d$ is bounded below by the second fundamental form at the corresponding projection point, which is also positive by strict convexity.
    
    Since $f(d)$ is strictly convex, by Lemma \ref{gluing lemma}, if we set
    \[
    {u}^*_1(\xi) = 
    \begin{cases}
    h(\xi), & \xi \in K \setminus \Omega, \\
    M_{\eta}(h(\xi), f(d(\xi))), & \xi \in \Omega, \\
    f(d(\xi)), & \xi \in \left\{ 0 \leqslant d \leqslant l \right\}.
    \end{cases}
    \]
    then ${u}^*_1$ is a negative smooth strictly convex function.

    Now we want to glue ${u}^*_1$ with ${u}^*_0$ to obtain the desired function. Since $g$ is increasing and negative, we can choose $\mu$ sufficiently large, such that 
    \[
    f(l_1) \geqslant u^*_0(\xi) + 2\eta, \quad \text{on } P_1 = \left\{ d = l_1 \right\},
    \]
    for some $0 < l_1 < l$. Then since $u^*_0(\xi) = v_0(P^{-1}(\xi))w^*(\xi)$ and $v_0$ is uniformly bounded, $u^*_0$ converges to $0$ uniformly. Thus we can choose $r$ close to $1$ such that 
    \[
    u^*_0(\xi) \geqslant f(d)(\xi) + 2\eta, \quad \text{on } P_2 = \partial B_r.
    \]
    With the above conditions, set $\tilde{\Omega}$ be the region enclosed by $P_1$ and $P_2$. By using Lemma \ref{gluing lemma}, let
    \[
    {u}^*_2(\xi) = 
    \begin{cases}
    {u}^*_1(\xi), & \xi \in \left\{ d \leqslant l_1 \right\}, \\
    M_{\eta}({u}^*_1(\xi), {u}^*_0(\xi)), & \xi \in \tilde{\Omega}, \\
    {u}^*_0(\xi), & \xi \in B_1 \setminus B_r.
    \end{cases}
    \]
    This is the desired function.
\end{proof}

Then setting $\tilde{v}_0 = \frac{{u}^*_2}{w^*}$, we obtain the desired function in hyperboloid. We conclude the main results of this section in the following: 

\begin{proposition}
Suppose $N_0=\{(x,h_0(x))\in \mathbb{R}^{n,1}; x\in \Omega \}$ is a piece of graphic spacelike strictly convex hypersurface and $Dh_0(\Omega) \subset \subset B_1$ is a strictly convex bounded open subset. There exits some constant  $c_* = c_*(h^*_0)$  as in Lemma \ref{13.3} only depending on $h^*_0$, where $h^*_0$ is the Legendre transform of $h_0$. Then for any $c \geqslant c_*$, we translate $N_0$ upward by $c$, which we still denote by $N_0$. Let $\varphi \in C^0(\mathbb{S}^{n-1})$ and $\varphi<0$. If $k = 1$, $N_0$ can be extended to an entire graphic spacelike strictly convex hypersurface $M_0$ satisfying all requirements of the initial hypersurface in Theorem \ref{inverse Hessian curvature flow}  and its support function has  boundary value  $\varphi $ on $\mathbb{H}^n(\infty)$.
For $2 \leqslant k \leqslant n$, if we further require $\varphi \equiv -1$,  such extension $M_0$ of $N_0$ also exists. 
\end{proposition}
\begin{proof}
% For $k=n$, the existence of extension is a corollary of Lemma \ref{13.3}. 
% For $k\neq n$,
By Lemma \ref{13.3}, we only need to bound $F$ from above near infinity. Since $F = \left( \frac{\sigma_n}{\sigma_{n-k}} \right)^{\frac{1}{k}}$, we have
\[
F^{ii} = \frac{F}{k}\left( \frac{1}{\lambda_i} - \frac{\sigma_{n-k-1}(\lambda|i)}{\sigma_{n-k}} \right).
\]
Take summation, we have
\begin{align*}
    \sum_i F^{ii} &= \frac{F}{k}\left( \frac{\sigma_{n-1}}{\sigma_n} - \left( k+1 \right)\frac{\sigma_{n-k-1}}{\sigma_{n-k}} \right) \leqslant \frac{F}{k} \frac{\sigma_{n-1}}{\sigma_n} \\
    &\overset{k=1}{=} 1.
\end{align*}
Then by \eqref{v_0}, we let $A = \Lambda_{v}$, $B = \Lambda_{\mu g}$. By concavity of $F$, we have
\begin{align*}
    F(A+B) &\leqslant F(A) + F^{ii}(A)B_{ii} \\
    &\overset{k=1}{\leqslant} C,
\end{align*} 
since $\lambda(B) \to 0$ as $\rho \to +\infty$. Thus, for $k=1$, we have proved the existence.  

For $2\leqslant k\leqslant n$, since $\varphi \equiv -1$, i.e. the unique solution $v$ of  \eqref{sigmak selfshrinker} is $v \equiv -1$. Therefore $\Lambda_{v}$ is the standard metric of $\mathbb{H}^n$. Thus,  
$F(\Lambda_v+\Lambda_{\mu g})$ is also bounded, which finishes our proof. 
\end{proof}

\begin{remark}
The upward translation of $N_0$ is necessary, since we require the initial hypersurface $M_0$ satisfies
\[
u_0(x) \rightarrow |x| \quad \text{as } |x| \rightarrow \infty.
\]
Consequently, if $N_0$ lies in the below of the light cone, then the spacelike condition gives that any strictly convex extension  should  also lies in the below of the light cone at infinity, which contradicts to the above asymptotic condition.    

\end{remark}

\section{A counterexample}

In the last section, we prove Theorem \ref{Counterexample}. For a support function \(v\) defined on  \(\H^n\), recall the definition of principal radii tensor:  
\[ \Lambda_v = \nabla^2 v - v\bar{g}, \]
where \(\bar{g}\) is the standard metric on \(\H^n\) and \(\nabla\) is the Levi-Civita connection of \(\bar{g}\). 
Let \(w^*(\xi) = \sqrt{1-|\xi|^2}\) and \(u^*(\xi) = v(P^{-1}(\xi))\cdot w^*(\xi)\). Define 
\[ b_{ij} = w^*\cdot \gamma^*_{ik}(\partial_{\xi_i \xi_j}u^*) \gamma^*_{kj}, \]
where \(\gamma^*_{ij} = \delta_{ij} - \frac{\xi_i\xi_j}{1+w^*}\). The matrix \((\gamma^*_{ij})\) is the square root of the matrix \((g_{ij}) = \delta_{ij} - \xi_i \xi_j\). 
\par It is known that the eigenvalues of the matrix \((b_{ij})\) are also the eigenvalues of $\Lambda_v$. We first consider the following eigenvalue equation in \(B_1\): 
\begin{equation}\label{eigenequ}
    \sigma_1(b_{ij}[u^*])(\xi) = \frac{\lambda u^*}{w^*}(\xi), \quad \text{in }\xi\in B_1,
\end{equation}
for \(-n \leqslant \lambda \leqslant 0\).

\subsection{The eigenvalue equation on \texorpdfstring{\(B_1\)}{B1}}\label{The eigenvalue equation on B1}

In this section we investigate \eqref{eigenequ}. By definitions of \(b_{ij}\) and \(w^*(\xi)\), the above equation can be written as  
\begin{equation}\label{eq:eigenvalue-equation-u*}
    \sqrt{1-|\xi|^2}\left(\Delta u^*(\xi) - \sum_{i,j=1}^{n}\xi_i \xi_j \partial_{ij}u^*(\xi)\right) = \frac{\lambda u^*(\xi)}{\sqrt{1-|\xi|^2}}, \quad \xi\in B_1 \quad(-n \leqslant \lambda \leqslant 0).
\end{equation}
To obtain a special solution to \eqref{eq:eigenvalue-equation-u*}, we consider the following form 
\begin{equation}\label{eq:expected-form}
    u^*(\xi) = a(r)(n\cos^2\theta - 1),
\end{equation}
where \(r = |\xi|\), and \(\theta\) is the angle between \(\xi\) and the \(\xi_1\)-axis. Denote \(Y(\theta) = n\cos^2\theta - 1\), then it is easy to check \(\Delta_{S^{n-1}}Y = -2n Y\). By the expression of Laplacian operator in the polar coordinate, we have 
\[\begin{aligned} 
    \Delta u^* &= \partial_{rr}u^* + \frac{n-1}{r}\partial_{r}u^* + \frac{1}{r^2}\Delta_{S^{n-1}}u^* \\
    &= \left[a''(r) + \frac{n-1}{r}a'(r) - \frac{2n}{r^2}a(r)\right]Y(\theta).
\end{aligned} \]
Then equation \eqref{eq:eigenvalue-equation-u*} reduces to 
\begin{equation}\label{eq:reduced-eigenvalue-equation-a}
    \mathcal{L}(a) := (1-r^2)a''(r) + \frac{n-1}{r}a'(r) - \left( \frac{2n}{r^2} + \frac{\lambda}{1-r^2} \right)a(r) = 0, \quad r\in (0,1).
\end{equation}
Assume that \(a(r) = r^2(1-r^2)^\beta \phi(r^2)\) for some constant \(\beta\), and some function \(\phi = \phi(s)\) to be determined later. To simplify the calculation, let \(s = r^2\) and \(f(s) = s(1-s)^\beta \phi(s)\). Then  
\begin{equation}\label{eq:chain-rule}
    a'(r) = 2rf'(s), \quad a''(r) = 2f'(s) + 4s f''(s). 
\end{equation} 
Substitute \eqref{eq:chain-rule} into \eqref{eq:reduced-eigenvalue-equation-a}, we get 
\begin{equation}\label{eq:La-1}
\begin{aligned}
    \mathcal{L}(a) &= (1-s)\left[2f'(s) + 4sf''(s)\right] + \frac{n-1}{r}\left[2r f'(s)\right] - \left( \frac{2n}{s} + \frac{\lambda}{1-s} \right)f(s) \\
    &= 4s(1-s)f''(s) + 2(n-s)f'(s) - \left( \frac{2n}{s} + \frac{\lambda}{1-s} \right)f(s).    
\end{aligned}
\end{equation} 
Since \(f(s) = s(1-s)^\beta \phi(s)\), we have 
\begin{equation}\label{eq:derivatives-f}
    \begin{aligned}
        f'(s) &= f\cdot\left(\frac{1}{s} - \frac{\beta}{1-s} + \frac{\phi'}{\phi}\right), \\
        f''(s) &= f'\cdot\left(\frac{1}{s} - \frac{\beta}{1-s} + \frac{\phi'}{\phi}\right) + f\left(-\frac{1}{s^2} - \frac{\beta}{(1-s)^2} + \frac{\phi''}{\phi} - \frac{(\phi')^2}{\phi^2}\right).
    \end{aligned}
\end{equation}
Substitute \eqref{eq:derivatives-f} into \eqref{eq:La-1}, we have 
\begin{equation}\label{eq:La-2}
    \begin{aligned}
        \mathcal{L}(a) &= 4s(1-s)^{\beta} \Big\{ s(1-s) \phi''(s) + \frac{1}{2}\left[ n + 4 - (4\beta+5)s \right] \phi'(s) \\
        &\quad + \frac{2(2\beta+1)(\beta+1)s - (2\beta n + 8\beta + \lambda + 2)}{4(1-s)} \phi(s) \Big\}.
    \end{aligned}
\end{equation}
We choose \(\beta\) such that the coefficient of \(\phi(s)\) reduces to a constant. To be specific, we choose \(\beta\) satisfying 
\begin{equation*}
    2(2\beta+1)(\beta+1) - (2\beta n + 8\beta + \lambda + 2) = 0,
\end{equation*}
which simplifies to 
\begin{equation}\label{eq:eq-beta}
    4\beta^2 - 2(n+1)\beta - \lambda = 0.
\end{equation}
We choose the smaller positive root of \eqref{eq:eq-beta} 
\begin{equation}\label{eq:choice-beta}
    \beta = \frac{(n+1) - \sqrt{(n+1)^2 + 4\lambda}}{4} \in [0,\frac{1}{2}],
\end{equation}
since $\lambda \in [-n, 0]$. 
Then \eqref{eq:La-2} reduces to 
\begin{equation}\label{eq:La-3}
    \begin{aligned}
        \mathcal{L}(a) &= 4s(1-s)^{\beta} \Big\{ s(1-s) \phi''(s) + \frac{1}{2}\left[ n + 4 - (4\beta+5)s \right] \phi'(s) \\
        &\quad - \frac{1}{2}(2\beta+1)(\beta+1)\phi(s) \Big\}.
    \end{aligned}
\end{equation}
Comparing \eqref{eq:La-3} with the standard hypergeometric equation 
\begin{equation}
    s(1-s)\phi'' + \big[c - (a+b+1)s\big]\phi' - ab\phi = 0,
\end{equation}
we get 
\begin{equation}
    a = \beta + 1, \quad b=\beta + \frac{1}{2}, \quad c = \frac{n+4}{2}.
\end{equation}
With the choice \eqref{eq:choice-beta}, we have \(\phi(s) = {}_2F_1\left(\beta + 1, \beta + \frac{1}{2}; \frac{n+4}{2} ; s\right)\). Here 
\[ {}_2F_1(a, b; c; z) = \sum_{k=0}^\infty \frac{(a)_k (b)_k}{(c)_k k!} z^k , \quad (a)_k := \begin{cases}
    1 ,\quad k=0 \\ a(a+1)\cdots(a+k-1),\quad k\geqslant 1 
\end{cases}  \]
is the Gauss hypergeometric function. Then a solution to \eqref{eq:reduced-eigenvalue-equation-a} is given by 
\begin{equation}\label{eq:choice-of-a}
    a(r) = r^2(1-r^2)^\beta {}_2F_1\left(\beta + 1, \beta + \frac{1}{2}; \frac{n+4}{2} ; r^2\right) \quad r\in (0,1),
\end{equation}
where \(\beta\) is a positive constant given by \eqref{eq:choice-beta}.
% \begin{remark}
%     This solution, however, tends to \(+\infty\) as \(r\to 1^-\). This fact differs from that when \(\lambda \leq 0\).
% \end{remark}
\par Since \(c-a-b = \frac{n+1}{2} - 2\beta > 0\), from the theory of hypergeometric functions we know that \(\lim_{s\to 1^-}{}_2F_1(a,b;c;s)\) is finite. This shows \(\lim_{r\to 1^-}a(r) = 0\). Specifically, by Gauss's hypergeometric formula:
\begin{equation}
    \begin{aligned}
        \lim_{s\to 1^-}{}_2F_1(a,b;c;s) &= \frac{\Gamma(c)\Gamma(c-a-b)}{\Gamma(c-a)\Gamma(c-b)} \\
        &= \frac{\Gamma(\frac{n+4}{2})\Gamma(\frac{n+1}{2} - 2\beta)}{\Gamma(\frac{n+2}{2}-\beta)\Gamma(\frac{n+3}{2}-\beta)} < +\infty.
    \end{aligned}
\end{equation}

It is clear that with the radial function \eqref{eq:choice-of-a} and the choice \eqref{eq:choice-beta}, the expected solution \(u^*(\xi)\) \eqref{eq:expected-form} satisfies  
\begin{equation}\label{minus}
    (b_{ij}[u^*])(0) = (D^2 u^*)(0) = \begin{pmatrix}
        2(n-1) &   &   &  \\
               & -2&   &  \\
               &   & \ddots&  \\
               &   &   & -2
    \end{pmatrix},
\end{equation}
which means \(\xi=0\) is a saddle point of \(b_{ij}[u^*](\xi)\).

\subsection{Proof of Theorem \ref{Counterexample}}

\par (1) We let \(\varepsilon > 0\) be small and consider
\begin{equation}\label{eq:v} 
v(y,t) = -e^{-nt} + \varepsilon \hat{v}(y),\quad y\in \H^n, t\in(0,+\infty),
\end{equation}
where \(\hat{v}(y)\) satisfies the equation 
\begin{equation}\label{eq:eigenvalue-eq-v} 
    \sigma_1(\Lambda_{\hat{v}}) = 0,\quad \text{in } \H^n. 
\end{equation}
By the above subsection, we can give a specific $\hat{v}$, which is the solution of \eqref{eigenequ} for $\lambda=0$.  
Then \(v(y,t)\) is a solution to 
\[ \pd{v}{t} = \sigma_1(\Lambda_v),  \]
which is the dual equation of \eqref{IQCF}. Namely, we can construct a family of entire spacelike strictly convex hypersurfaces $\mathcal{M}_t$ satisfying \eqref{IQCF} such that the support function of $\mathcal{M}_t$ is $v$.

By  \eqref{minus},  \(\Lambda_{\hat{v}}(y)\) has a negative eigenvalue \(\lambda_i = -2\) at some point \(y_0\), then \(\Lambda_v(y_0)\) will have an eigenvalue \(\lambda_i = e^{-nt}-2\varepsilon \). It will become \(0\) at $T=-\frac{\log 2\varepsilon}{n}$ if we take \(\varepsilon < 1/2\).

In the next, we will prove $v$ is strictly hyperbolic convex near the infinity for any $t\leqslant T$. In our case, $\hat{v}$ gives $\beta=0$ and the function in \eqref{eq:choice-of-a} becomes
\[
a(r)=r^2H(r^2) \text{ and } \hat{u}^*(\xi)=a(r)Y(\theta),\]
where
\[
H(s):={}_2F_1\left(1,\frac12;\frac{n+4}{2};s\right). 
\]
Therefore, in the polar coordinate, $\{\partial_r,\partial_\theta,\cdots\}$, the matrix $b$ has the block form
\[
b=
\begin{pmatrix}
b_{rr}&b_{r\theta}&0\\
b_{r\theta}&b_{\theta\theta}&0\\
0&0&b_T I_{n-2}
\end{pmatrix},
\]
where
\[
% b_{rr}=w^3\left(2H+10r^2H'+4r^4H''\right)(n\cos^2\theta-1),
b_{rr}=2w^*\left[H-(n-1)r^2H'\right](n\cos^2\theta-1),
\]
\[
b_{r\theta}=-2n(w^*)^2\sin\theta\cos\theta\left(H+2r^2H'\right),
\]
\[
b_{\theta\theta}=2w^*\left[H(n\sin^2\theta-1)+r^2H'(n\cos^2\theta-1)\right],
\]
\[
b_T=2w^*\left[-H+r^2H'(n\cos^2\theta-1)\right].
\]
Consequently, the matrix $b$ can be factored as
\[
b=w^*\mathcal B.
\]
By Gauss' formula for the hypergeometric function, both $H$ and $H'$ extend continuously to $s=1$. More precisely,
\[
H(1)=\frac{n+2}{n+1},
\quad
H'(1)=\frac{n+2}{(n-1)(n+1)}.
\]
Hence there exists a constant $C>0$ such that
\[
\|\mathcal B\|\leqslant C
\]
for $r$ sufficiently close to $1$. Since $w^*=\sqrt{1-r^2}\to0$ as $r\to1^-$, we obtain
\[
\|b\|=w^*\|\mathcal B\|\leqslant C\sqrt{1-r^2}\to 0.
\]
Thus all the eigenvalues of the principal radii matrix $b$ converge to zero at the boundary of the Klein ball. 
Therefore, $\lambda_v$ is strictly hyperbolic convex for $t\leqslant T$  near the infinity. If $\hat{v}$ has another saddle point $y_1$ except $y_0$ and some eigenvalue of $\Lambda_{\hat{v}}$ is less than $-2$, we only need  take a smaller $\varepsilon$. 

At last, let \(U^*(\xi, t) = v(P^{-1}(\xi), t)\cdot w^*(\xi)\), we will show that the Legendre transform of $U^*$ is defined on $\mathbb{R}^n$ for any $t$, i.e.
\[
D U^*(B_1) = \mathbb{R}^n.
\]
% Let 
% \[
% e_r=\partial_r,\quad e_\theta=\frac1r\partial_\theta.
% \]
By direct calculation, we have
\[
D \hat{u}^*=2r\left(H+r^2H'\right)Y(\theta)e_r-2nrH\sin\theta\cos\theta e_\theta.
\]
Hence
\[
\sup_{\xi\in B_1}|D\hat{u}^*(\xi)|<\infty.
\]
However, since
\[
D U^*= e^{-nt}\frac{\xi}{\sqrt{1-r^2}} + \varepsilon D \hat{u}^*.
\]
We have for $t \leqslant T$, for any $B_R \subset \mathbb{R}^n$, there exists some $r_0 < 1$, such that 
\[
D U^*(B_{r_0}) \supset B_R.
\]
% \[
% D U^* = e^{-nt}\frac{\xi}{\sqrt{1-r^2}} + O(1),
% \]
% and
% \[
% \left| D U^* \right| \to \infty, \quad \text{as } r \to 1^-.
% \]
Therefore we have $D U^*(B_1) = \mathbb{R}^n$. Hence Theorem \ref{Counterexample} (1) is proved.

\begin{remark}
    Notice that in our construction, $\hat{v} = \frac{\hat{u}^*}{w^*}$ is unbounded at infinity. Indeed, for any solution of  \eqref{eq:eigenvalue-eq-v}, if we  require it is bounded on $\mathbb{H}^n(\infty)$, it should be identically zero. We can prove this  for arbitrary eigenvalue equation on $\mathbb{H}^n$:
    \[
    \Delta v = \lambda v,
    \]
    for constant $\lambda>0$. Its solution must be identically zero by applying Omori-Yau's generalized maximum principle if $v \in C^2(\mathbb{H}^n) \cap C^0(\overline{\mathbb{H}}{}^n)$.  
\end{remark}

%\subsection{The endpoint case \texorpdfstring{$\lambda=-n$}{lambda=-n}}
(2) The dual equation of \eqref{QS} is
\begin{equation}
    \sigma_1(\Lambda[v]) = -n v,\quad \text{in } \H^n, 
\end{equation}
where $v$ is the support function.  It can be rewritten as \eqref{eigenequ} for $\lambda=-n$. As before, we let
$$v=-1-\alpha \hat{v},$$
where $\hat{v}$ is constructed in the previous subsection.  
In this case, $\hat{v}$ gives $\beta=\frac{1}{2}$ and the function in \eqref{eq:choice-of-a} becomes 
\begin{equation}
    a(r) = r^2(1-r^2)^\frac{1}{2} H(r^2),
\end{equation}
where
\[
H(s):={}_2F_1\left(\frac{3}{2}, 1; \frac{n+4}{2}; s\right).
\]
Similarly, the matrix $b$ has the block form
\[
b=
\begin{pmatrix}
b_{rr}&b_{r\theta}&0\\
b_{r\theta}&b_{\theta\theta}&0\\
0&0&b_T I_{n-2}
\end{pmatrix},
\]
where
\[
b_{rr}=\left[(2-3r^2)H-2(n-1)r^2(1-r^2)H'\right](n\cos^2\theta-1),
\]
\[
b_{r\theta}=-2nw^*\sin\theta\cos\theta\left[(1-2r^2)H+2r^2(1-r^2)H'\right],
\]
\[
b_{\theta\theta}=-r^2H(n\cos^2\theta-1) + 2(1-r^2)\left[H(n\sin^2\theta-1)+r^2H'(n\cos^2\theta-1)\right],
\]
\[
b_T=-r^2H(n\cos^2\theta-1) + 2(1-r^2)\left[-H+r^2H'(n\cos^2\theta-1)\right].
\]
Since we have
\[
H'(s) = \frac{3}{n+4} {}_2F_1\left(\frac{5}{2}, 2; \frac{n+6}{2}; s\right).
\]
For any $n \geqslant 2$, as $r\to1^-$, we have
\[
(1-r^2)H' \to 0. 
\]
Hence we have
\[
b \to -H(1)(n\cos^2\theta-1)I \quad \text{as } r\to1^-,
\]
where $H(1) = \frac{n+2}{n-1}$ and $I$ is the identity matrix. Thus the eigenvalues of $b$ are uniformly bounded in $B_1$. The hyperbolic Hessian of $v$ is 
\[
I - \alpha b.
\]
At the boundary, we require
\[
1 + \alpha H(1)(n\cos^2\theta-1) > 0,
\]
At origin, we require
\[
1 - 2\alpha (n-1) \leqslant 0.
\]
If $n \geqslant 3$, both conditions hold for some $\alpha$, for example, $\alpha = \frac{1}{2(n-1)}$. 
Moreover, we can see that $\hat{v} = H(1)Y(\theta)$ on $\mathbb{H}^n(\infty)$, thus $\hat{v}$ is bounded. 

Same as in (1), we also let \(U^*(\xi, t) = v(P^{-1}(\xi), t)\cdot w^*(\xi)\), and we will show that 
\[
D U^*(B_1) = \mathbb{R}^n.
\]
By direct calculation, we have
\[
D \hat{u}^*=\frac{r}{w^*}\left[(2-3r^2)H+2r^2(1-r^2)H'\right]Y(\theta)e_r-2nrw^*H\sin\theta\cos\theta e_\theta.
\]
Hence as $r \to 1^-$, 
\begin{align*}
    D U^*(r, \theta) &= \frac{r}{w^*}e_r - \alpha D \hat{u}^* \\
    &\to \frac{1}{w^*}\left[1 + \alpha H(1)Y(\theta) \right] e_r + o(\frac{1}{w^*}).
\end{align*}
By the choice of $\alpha$, we have
\[
1 + \alpha H(1)Y(\theta) > 0.
\]
Therefore
\[
\left| D U^* \right| \to \infty, \quad \text{as } r \to 1^-.
\]
From this we know the set $D U^*(B_1)$ is both open and closed, thus $D U^*(B_1) = \mathbb{R}^n$. Hence Theorem \ref{Counterexample} (2) is proved.

\appendix

\section{Proof of Lemma \ref{Gauss selfshrinker Pogorelov}}\label{Proof of Gauss selfshrinker Pogorelov}

\begin{proof}
    Take test function
    \begin{align*}
        W &= \left( -\varepsilon x_{n+1}(y) - v(y) \right)^{\alpha} \lambda_1 e^{\frac{\beta}{2} |\nabla v|^2},
    \end{align*}
    where $\alpha, \beta$ are two positive constants to be determined later.

    Suppose $W$ achieves its maximum value at some point $y_0 \in U_r$. Choose a local orthonormal frame $\{\tau_1, \cdots, \tau_n\}$ around $y_0$ such that $\Lambda_{ij} = \lambda_i\delta_{ij}$ and $\lambda_1 \geqslant \lambda_2 \geqslant \cdots \geqslant \lambda_n$. Suppose $\lambda_1$ has multiplicity $m$ at $y_0$. Differentiating $W$ twice, at $y_0$, we have
    \begin{align}\label{Gauss1}
        0 = (\log W)_i =& \alpha \frac{\varepsilon (x_{n+1})_i + v_i}{\varepsilon x_{n+1} + v} + \frac{(\lambda_1)_i}{\lambda_1} + \beta v_iv_{ii} \\
        0 \geqslant (\log W)_{ii} 
        =& \alpha \left( \frac{\varepsilon x_{n+1} + v_{ii}}{\varepsilon x_{n+1} + v} - \frac{(\varepsilon (x_{n+1})_{i} + v_{i})^2}{(\varepsilon x_{n+1} + v)^2}  \right)
        \notag \\
        &+ \frac{(\lambda_1)_{ii}}{\lambda_1} - \frac{(\lambda_1)_{i}^2}{\lambda_1^2} 
        + \beta \left( v_{ii}^2 + \sum_k v_kv_{kii} \right) 
        \notag \\
        =& \alpha \left( 1 + \frac{\lambda_i}{\varepsilon x_{n+1} + v} \right) - \frac{1}{\alpha}\left( \frac{(\lambda_1)_i}{\lambda_1} + \beta v_{i}v_{ii} \right)^2 
        \notag \\
        &+ \frac{(\lambda_1)_{ii}}{\lambda_1} - \frac{(\lambda_1)_{i}^2}{\lambda_1^2} 
        + \beta \left( v_{ii}^2 + \sum_k v_kv_{kii} \right) 
        \notag \\
        \geqslant& 
        \frac{\alpha \lambda_i}{\varepsilon x_{n+1} + v} 
        - \left( \frac{C\beta}{\alpha^2}\frac{(\lambda_1)_{i}^2}{\lambda_1^2} + \frac{\beta}{2}v_{ii}^2 \right) - \frac{\beta^2}{\alpha}v_{i}^2v_{ii}^2 
        \notag \\
        &+ \frac{(\lambda_1)_{ii}}{\lambda_1} - \left( 1+\frac{1}{\alpha} \right)\frac{(\lambda_1)_{i}^2}{\lambda_1^2} 
        + \beta \left( v_{ii}^2 + \sum_k v_kv_{kii} \right). 
    \end{align}
    By Lemma 5 in \cite{brendle2017asymptotic}, we have
    \begin{align}
        \delta_{kl} \left( \lambda_1 \right)_i &= \Lambda_{kl,i} \quad 1 \leqslant k,l \leqslant m, \label{lambda 1st order} \\
        \left( \lambda_1 \right)_{ii} &\geqslant \Lambda_{11,ii} + 2 \sum_{p > m} \frac{\Lambda^2_{1p,i}}{\lambda_1 - \lambda_p}. 
    \end{align}
    Plugging into \eqref{Gauss1}, we have
    \begin{align}\label{Gauss2}
        0 \geqslant& 
        \frac{\Lambda_{11,ii}}{\lambda_1} + 2 \sum_{p > m} \frac{\Lambda^2_{1p,i}}{\lambda_1(\lambda_1 - \lambda_p)} - \left( 1+\frac{1}{\alpha} \right)\frac{\Lambda_{11,i}^2}{\lambda_1^2} 
        \notag \\
        &+ \frac{\alpha \lambda_i}{\varepsilon x_{n+1} + v} 
        - \left( \frac{C\beta}{\alpha^2}\frac{\Lambda_{11,i}^2}{\lambda_1^2} + \frac{\beta}{2}v_{ii}^2 \right) - \frac{\beta^2}{\alpha}v_{i}^2v_{ii}^2 
        + \beta \left( v_{ii}^2 + \sum_k v_kv_{kii} \right).
    \end{align}
    Contracting with $\sigma_n^{ii}$, we have
    \begin{align}
        0 \geqslant& 
        \frac{\left( f(v) \right)_{11} - \sigma_n^{pq,rs} \Lambda_{pq,1}\Lambda_{rs,1}}{\lambda_1} + \frac{\sigma_n^{ii}(\lambda_i - \lambda_1)}{\lambda_1} + 2 \sum_{p > m} \frac{\sigma_n^{ii}\Lambda^2_{1p,i}}{\lambda_1(\lambda_1 - \lambda_p)} 
        \notag \\
        &- \left( 1+\frac{1}{\alpha}+\frac{C\beta}{\alpha^2} \right)\sigma_n^{ii} \frac{\Lambda_{11,i}^2}{\lambda_1^2} + \frac{\alpha n \sigma_n}{\varepsilon x_{n+1} + v} 
        - \frac{\beta^2}{\alpha} \sigma_n^{ii}v_{i}^2v_{ii}^2 
        \notag \\
        &+ \frac{\beta}{2} \sigma_n^{ii}v_{ii}^2 
        + \beta v_i \left( f(v) \right)_i 
        + \beta \sigma_n^{ii} v_i^2,
    \end{align}
    where by differentiating equation $\sigma_n \left[v_{ij}-v\delta_{ij}\right] = (-v)^n \eqqcolon f(v)$ twice, we have used 
    \begin{align}
        \sigma_n^{ii}\Lambda_{ii,k} = \sigma_n^{ii}\Lambda_{ki,i} = \left( f(v) \right)_k,
    \end{align}
    and 
    \begin{align}
        \sigma_n^{ii}\Lambda_{ii,11} &= \left( f(v) \right)_{11} - \sigma_n^{pq,rs} \Lambda_{pq,1}\Lambda_{rs,1}.
    \end{align}
    We also use the fact that
    \begin{align}
        \Lambda_{11,ii} = \Lambda_{ii,11} + \Lambda_{ii} - \Lambda_{11}.
    \end{align}
    For $f$, we have
    \begin{align}
        \left( f(v) \right)_k &= n(-v)^{n-1}(-v_k), \\
        \left( f(v) \right)_{11} &= n(n-1)(-v)^{n-2}v_1^2 + n (-v)^{n-1}(-v_{11}) \geqslant -C\lambda_1.
    \end{align}
    Hence
    \begin{align}
        0 \geqslant& 
        \frac{-\sigma_n^{pp,qq} \Lambda_{pp,1}\Lambda_{qq,1} - \sum_{p \neq q} \sigma_n^{pq,qp} \Lambda_{pq,1}^2}{\lambda_1} 
        - \sigma_{n-1}
        + 2 \sum_{p > m} \frac{\sigma_n^{ii}\Lambda^2_{1p,i}}{\lambda_1(\lambda_1 - \lambda_p)} 
        \notag \\
        &- \left( 1+\frac{1}{\alpha}+\frac{C\beta}{\alpha^2} \right)\sigma_n^{ii} \frac{\Lambda_{11,i}^2}{\lambda_1^2} + \frac{\alpha n \sigma_n}{\varepsilon x_{n+1} + v} 
        - \frac{\beta^2}{\alpha} \sigma_n^{ii}v_{i}^2v_{ii}^2 
        \notag \\
        &+ \frac{\beta}{2} \sigma_n^{ii}v_{ii}^2 
        % + \beta \sigma_n^{ii} v_i^2 
        -C \beta.
    \end{align}
    Same as the calculation of Lemma 4.1 in \cite{lu2025interior}, we have
    \begin{align*}
        -\sum_{p \neq q} \frac{\sigma_n^{pq,qp}\Lambda_{pq,1}^2}{\lambda_1} \geqslant -2 \sum_{i>m} \frac{\sigma_n^{i1,1i} \Lambda_{i1,1}^2}{\lambda_1} = 2 \sum_{i>m} \frac{(\sigma_n^{ii} - \sigma_n^{11}) \Lambda_{i1,1}^2}{\lambda_1(\lambda_1 - \lambda_i)}. 
    \end{align*}
    and
    \begin{align*}
        2 \sum_{i} \sum_{p > m} \frac{\sigma_n^{ii} \Lambda^2_{1p,i}}{\lambda_1 \left( \lambda_1 - \lambda_p \right)} &\geqslant 2 \sum_{p > m} \frac{\sigma_n^{11} \Lambda^2_{1p,1}}{\lambda_1 \left( \lambda_1 - \lambda_p \right)}.
    \end{align*}
    Therefore we have
    \begin{align}\label{Gauss3}
        0 \geqslant& 
        \frac{-\sigma_n^{pp,qq} \Lambda_{pp,1}\Lambda_{qq,1}}{\lambda_1} 
        - \sigma_{n-1}
        + 2 \sum_{i > m} \frac{\sigma_n^{ii}\Lambda^2_{i1,1}}{\lambda_1^2} 
       - \left( 1+\frac{1}{\alpha}+\frac{C\beta}{\alpha^2} \right)\sigma_n^{ii} \frac{\Lambda_{11,i}^2}{\lambda_1^2}  \notag \\
        &+ \frac{\alpha n \sigma_n}{\varepsilon x_{n+1} + v} 
        - \frac{\beta^2}{\alpha} \sigma_n^{ii}v_{i}^2v_{ii}^2 
        + \frac{\beta}{2} \sigma_n^{ii}v_{ii}^2  
        % + \beta \sigma_n^{ii} v_i^2 
        -C \beta.
    \end{align}
    Set $\left\{ \xi_i = \Lambda_{ii,1} \right\}$. Then, by the concavity inequality for $\sigma_n$ operator, we obtain
    \[
    -\sum_{i \neq j} \sigma_n^{ii,jj} \xi_i \xi_j + \frac{2}{\sigma_n} \left( \sum_i \sigma_n^{ii}\xi_i \right)^2 \geqslant \left( 1 + \frac{1}{n+1} \right) \frac{\sigma_n^{11}\xi_1^2}{\lambda_1}.
    \]
    Plugging into \eqref{Gauss3}, we have
    \begin{align}\label{Gauss4}
        0 \geqslant& 
        \left( 1 + \frac{1}{n+1} \right) \frac{\sigma_n^{11}\Lambda^2_{11,1}}{\lambda_1^2} 
        - \sigma_{n-1}
        + 2 \sum_{i > m} \frac{\sigma_n^{ii}\Lambda^2_{i1,1}}{\lambda_1^2} 
        \notag \\
        &- \left( 1+\frac{1}{\alpha}+\frac{C\beta}{\alpha^2} \right)\sigma_n^{ii} \frac{\Lambda_{11,i}^2}{\lambda_1^2} + \frac{\alpha n \sigma_n}{\varepsilon x_{n+1} + v} 
        - \frac{\beta^2}{\alpha} \sigma_n^{ii}v_{i}^2v_{ii}^2 
        + \frac{\beta}{2} \sigma_n^{ii}v_{ii}^2  
                -C \beta.
    \end{align}
    By \eqref{lambda 1st order}, we know $\Lambda_{1i,1} = 0$ for $1<i\leqslant m$, hence we can rewrite \eqref{Gauss4} as
    \begin{align}\label{Gauss5}
        0 \geqslant& 
        \left( \frac{1}{n+1} - \frac{1}{\alpha} - \frac{C\beta}{\alpha^2} \right) \frac{\sigma_n^{11}\Lambda^2_{11,1}}{\lambda_1^2} 
        + \left( 1 - \frac{1}{\alpha} - \frac{C\beta}{\alpha^2} \right) \sum_{i > m} \frac{\sigma_n^{ii}\Lambda^2_{i1,1}}{\lambda_1^2} 
        \notag \\
        &- \sigma_{n-1} + \frac{\alpha n \sigma_n}{\varepsilon x_{n+1} + v} 
        - \frac{\beta^2}{\alpha} \sigma_n^{ii}v_{i}^2v_{ii}^2 
        + \frac{\beta}{2} \sigma_n^{ii}v_{ii}^2 
        % + \beta \sigma_n^{ii} v_i^2 
        -C \beta.
    \end{align}
    If we assume $\left( -\varepsilon x_{n+1} - v \right)^{\alpha} \lambda_1$ is sufficiently large, we have
    \begin{align*}
        \frac{\alpha n \sigma_n}{\varepsilon x_{n+1} + v} \geqslant -C \lambda_1.
    \end{align*}
    Moreover, since $v_{ii} = v + \lambda_i$, we have
    \begin{align*}
        \sigma_n^{ii}v_{ii}^2 = \sigma_n^{ii}\left( v + \lambda_i \right)^2 \geqslant c_0 \sigma_{n-1} + c_0 \lambda_1 - C.
    \end{align*}
    Plugging into \eqref{Gauss5}, we have
    \begin{align}
        0 \geqslant& 
        \left( \frac{1}{n+1} - \frac{1}{\alpha} - \frac{C\beta}{\alpha^2} \right) \frac{\sigma_n^{11}\Lambda^2_{11,1}}{\lambda_1^2} 
        + \left( 1 - \frac{1}{\alpha} - \frac{C\beta}{\alpha^2} \right) \sum_{i > m} \frac{\sigma_n^{ii}\Lambda^2_{i1,1}}{\lambda_1^2} 
        \notag \\
        &+ \left( \frac{c_0}{4}\beta - 1 \right) \sigma_{n-1} 
        + \left( \frac{\beta}{4} - C \frac{\beta^2}{\alpha} \right)\sigma_n^{ii}v_{ii}^2 + \left( \frac{c_0}{4}\beta - C \right) \lambda_1 -C \beta.
    \end{align}
    Then if we choose $\alpha = \beta^2$ sufficiently large, we obtain
    \begin{align}
        C\beta \geqslant \left( \frac{c_0}{4}\beta - C \right) \lambda_1.
    \end{align}
\end{proof}

\section{Recover hypersurface from its support function}\label{Recover hypersurface from its support function}

% \subsection{Legendre transform}

Given a hyperbolic convex function $v$ defined on hyperboloid $\mathbb{H}^n$, where
\[
\mathbb{H}^n = \{X \in \mathbb{R}^{n,1} \mid \langle X, X \rangle = -1\}.
\]
Then define $u^*$ on $B_1$ by $u^*(\xi) = \sqrt{1 - |\xi|^2} v(y(\xi))$ where $y(\xi) = \frac{(\xi, 1)}{\sqrt{1 - |\xi|^2}}$. Moreover, let $u$ be the Legendre transform of $u^*$. Let $u = \mathscr{P}[v]$ by this procedure. Indeed
\begin{align*}
    u(x) \coloneqq \sup_{\xi \in B_1} \left\{ \langle x, \xi \rangle - \sqrt{1 - |\xi |^2} v(y(\xi)) \right\}.
\end{align*}

\begin{lemma}
    For $v$ solves \eqref{sigmak selfshrinker}, let $u^*=\sqrt{1 - |\xi |^2} v(y(\xi))$, then we have
    \[
    Du^*(B_1) = \mathbb{R}^n
    \]
    i.e. Legendre transform of $u^*$ is defined on the whole space $\mathbb{R}^n$.
\end{lemma}
\begin{proof}
    Let $a \coloneqq \sup \varphi < 0$, then $\tilde{v} \equiv a$ is a special solution to \eqref{sigmak selfshrinker} with constant boundary value. By the maximum principle, we have
    \[
    \tilde{v} \geqslant v \quad \text{in } \mathbb{H}^n
    \]
    Similarly let $\tilde{u}^*=\sqrt{1 - |\xi |^2} \tilde{v}(y(\xi))$, we have
    \[
    \tilde{u}^* \geqslant u^* \quad \text{in } B_1,
    \]
    and
    \[
    \tilde{u}^* = u^* = 0 \quad \text{on } \partial B_1.
    \]
    Thus we have
    \[
    Du^*(B_1) \supseteq D\tilde{u}^*(B_1). 
    \]
    However, $\tilde{v}$ corresponds to the hyperboloid
    \[
    \tilde{u}(x) = \sqrt{a^2+|x|^2}.
    \]
    Hence $D\tilde{u}^*(B_1) = \mathbb{R}^n$, which concludes the proof.
\end{proof}
Therefore when $v$ is the solution to Dirichlet problem \eqref{sigmak selfshrinker}, $u = \mathscr{P}[v]$ is defined on $\mathbb{R}^n$, and \[
\mathcal{M} = \{ X = (x,u(x)) \mid x \in \mathbb{R}^n \}
\]
is the desired entire spacelike strictly convex hypersurface.

\bibliographystyle{abbrv} 
\bibliography{references} 

\end{document}